\documentclass{amsart}
\usepackage{amsfonts,amssymb}
\usepackage{graphicx}
\usepackage{enumerate}
\usepackage{datetime}

  \newtheorem{theorem}{Theorem}
  \newtheorem{corollary}{Corollary}
  \newtheorem{proposition}{Proposition}
  \newtheorem{lemma}{Lemma}
\newtheorem{question}{Question}

\DeclareMathOperator{\lcm}{lcm}
\begin{document}

\title{The Schinzel-Szekeres function}

\begin{abstract}
We derive asymptotic estimates for distribution functions related to the Schinzel-Szekeres function. 
These results are then used in three different applications: the longest simple path in the divisor graph,
a problem of Erd\H{o}s about a sum of reciprocals, and the small sieve of Erd\H{o}s and Ruzsa.
\end{abstract}

\author{Andreas Weingartner}
\address{ 
Department of Mathematics,
351 West University Boulevard,
 Southern Utah University,
Cedar City, Utah 84720, USA}
\email{weingartner@suu.edu}

\maketitle

\section{Introduction}

The Schinzel-Szekeres function is defined as
\begin{equation*}
  F(n)=
  \begin{cases}
    1  & \quad (n=1) \\
    \max\{d\,P^-(d) : d|n, \ d>1\} & \quad (n\ge 2),
  \end{cases}
\end{equation*}
where $P^-(d)$ denotes the smallest prime factor of $d>1$ and $P^-(1):=\infty$.
This function appears  implicitly in \cite{SS}.
It plays a role in various applications, 
such as a problem of Erd\H{o}s \cite{Erd} considered by Schinzel and Szekeres \cite{SS}, 
the small sieve of Erd\H{o}s and Ruzsa \cite{Ru}, 
the distribution of divisors \cite{ EDD1, Ten86, PDD}, and the length of the longest simple path in the divisor graph \cite{AEDD, Ten95}. 
For the corresponding counting function
$$
A(x) = |\{n\ge 1: F(n)\le x\}|,
$$
Schinzel and Szekeres \cite{SS} showed that $A(x) = o(x)$ as $x\to \infty$. 
Ruzsa \cite{Ru} found that there is a constant $c>0$ such that $A(x) \le \frac{x}{\log^c x}$. 
Tenenbaum \cite{Ten86, Ten95} improved this to $A(x)=\frac{x}{\log x} (\log\log x)^{O(1)}$
and Saias \cite{AEDD} obtained $A(x) \asymp \frac{x}{\log x}$.

\begin{theorem}\label{thm1}
For $x\ge 2$, we have
\begin{equation*}\label{Aasymp}
A(x)=\frac{a x}{\log x}\left(1+O\left(\frac{1}{\log x}\right)\right),
\end{equation*}
where the constant  $a =1.53796...$ is given by
\begin{equation*}\label{adef}
 a=\frac{1}{1-e^{-\gamma}} \left( 1 -\gamma +\int_1^\infty A(t) g(t) \frac{dt}{t^2}  \right),
\end{equation*}
$\gamma$ is Euler's constant and
\begin{equation*}
 g(t) =\biggl(\sum_{p\le t }\frac{\log p}{p-1} +\gamma - \log t  \biggr)\prod_{p\le t} \left(1-\frac{1}{p}\right) .
\end{equation*}
\end{theorem}

\subsection{The longest simple path in the divisor graph}

The divisor graph of order $n$ consists of vertices $1,2,\ldots, n$ and an edge between vertices $j, k$ if and only if $j|k$ or $k|j$. 
Let $f(n)$ be the maximal number of vertices in a simple path of the divisor graph of order $n$. 
Improving on earlier work by Pollington \cite{Pol}, Pomerance \cite{Pom} and Tenenbaum \cite{Ten95}, Saias \cite{AEDD,EGD5} showed that $f(n) \asymp \frac{n}{\log n}$ and $f(n) \ge A(n/2)>0.37 \frac{ n}{\log n}$ for $n\ge n_0$.
With Theorem \ref{thm1}, we obtain the following lower bound.
\begin{corollary}
 For all sufficiently large $n$,
$$
f(n) > 0.76898 \frac{n}{\log n}.
$$
\end{corollary}

\subsection{A problem of Erd\H{o}s}

Erd\H{o}s \cite{Erd} considered the question of how large $\sum_{n\in \mathcal{S}} \frac{1}{n}$ can be for a set $\mathcal{S}$ of integers up to $x$ 
with the property 
\begin{equation}\label{lcmprop}
 m,n \in \mathcal{S} \text{ and }   n \neq  m  \ \Rightarrow \ \lcm(m,n) > x.
\end{equation}
Define
$$
R(x) = \max_{\mathcal{S}} \sum_{n\in \mathcal{S}} \frac{1}{n},
$$
where the maximum is taken over all sets of integers $\mathcal{S} \subset \{2,3,\ldots, \lfloor x\rfloor \}$ that have the $\lcm$ property \eqref{lcmprop}. 

Erd\H{o}s \cite{Erd}  proposed the problem to show that $R(x) <2$ and conjectured \cite{SS} that $R(x)<1+o(1)$ as $x\to \infty$. Lehman \cite{Leh} improved this to $R(x) < \frac{7}{6}+\frac{1}{6x}$. 
It seems that there are only two known cases where $R(x)>1$, namely $R(5)=\frac{31}{30}$ from $\mathcal{S}=\{2,3,5\}$,
and $R(11)=\frac{4699}{4620}$ from $\mathcal{S}=\{3,4,5,7,11\} $. 
Schinzel and Szekeres \cite{SS} showed that $R(x) \le \frac{31}{30}$, with equality being attained only at $x=5$. 
Moreover, they proved that $R(x)< c+o(1)$, where $c=1.017262...$.
They \cite[Theorem 3]{SS} also established the lower bound $R(x) >1 - o(1)$, by way of $\mathcal{S}=\mathcal{B}(x)$, 
where
$$ \mathcal{B}(x) = \{ n\le x : F(n)> x \text{ and } F(d)\le x \text{ for all } d|n, d<n\}.$$
It is not difficult to see  \cite[Proof of Thm. 3]{SS} that $\mathcal{B}(x)$ has the $\lcm$ property \eqref{lcmprop}. 

\begin{theorem}\label{corBrec}
For $x\ge 2$, we have
$$
\sum_{n\in \mathcal{B}(x)} \frac{1}{n} =1- \frac{\delta}{\log x} + O\left(\frac{1}{(\log x)^{3/2}}\right),
$$
where $\delta:= a+\gamma-1-\beta =0.560...$, the constant $a$ is as in Theorem \ref{thm1} and $\beta$ is given by \eqref{betadef}. 
More precisely, $0.560374<\delta < 0.560579$.  
\end{theorem}

Theorem \ref{corBrec} improves the estimate $ 1+O(\frac{1}{\log x})$ due to Saias \cite[Remark 1]{AEDD} and yields the improved lower bound
$$
R(x) > 1- \frac{\delta +o(1)}{\log x}, \quad \delta=0.560...
$$
We can improve this further by making a small change to $\mathcal{B}(x)$, based on the two known instances where $R(x)>1$.
Let $\mathcal{B}'(x)$ be the set obtained from $\mathcal{B}(x)$, by replacing every $n\in (\frac{x}{6},\frac{x}{5}]\cap \mathcal{B}(x)$
by $\{2n,3n,5n\}$, and every $n\in (\frac{x}{12},\frac{x}{11}]\cap \mathcal{B}(x)$ by $\{3n,4n,5n,7n,11n\}$.
Then $\mathcal{B}'(x) \subset [1,x]$ and $\mathcal{B}'(x)$ has the $\lcm$ property \eqref{lcmprop}.

\begin{theorem}\label{Rcor}
For $x\ge 2$, we have
$$
R(x) \ge  \sum_{n\in \mathcal{B}'(x)} \frac{1}{n} = 1- \frac{\kappa}{\log x} + O\left(\frac{1}{(\log x)^{3/2}}\right), $$
where $\kappa =0.543...$ is given by \eqref{kappadef}. More precisely, $0.543595<\kappa<0.543804$.
\end{theorem}
It is clear that any other examples of $R(x)>1$, besides $x=5,11$, if they exist, 
would lead to further small improvements. In the absence of such examples, it seems plausible 
that Theorem \ref{Rcor} is best possible. 
\begin{question}
Is it true that 
$R(x) = 1- \frac{\kappa +o(1)}{\log x} $ as $x\to \infty$ ?
\end{question}

\subsection{The small sieve of Erd\H{o}s and Ruzsa}

Here the question is how small we can make the number of unsieved integers up to $x$, 
after removing multiples of a set $\mathcal{S}$  with the property
\begin{equation}\label{SSER}
  \sum_{n\in \mathcal{S}} \frac{1}{n} \le 1, \quad 1\notin \mathcal{S}.
\end{equation}
Define
$$
H(x) = \min_{\mathcal{S}} |\{n\le x: s \nmid n \text{ for all } s\in \mathcal{S}  \}|,
$$
where the minimum is over all sets $ \mathcal{S}$ satisfying \eqref{SSER}.
Ruzsa \cite{Ru} showed that $\frac{x}{10\log x} \le H(x) < \frac{x}{(\log x)^c}$ for some constant $c>0$ and all $x\ge x_0$.
Tenenbaum \cite{Ten86} improved the upper bound to $H(x) \ll \frac{x}{\log x} (\log\log x)^2$, and Saias \cite{AEDD} 
established $H(x) \ll \frac{x}{\log x}$ via the inequality
$H(x) \le \max(A(x),B(x) +\sqrt{x})$, where $B(x) = |\mathcal{B}(x)|$.

Saias \cite{AEDD} asked whether $\mathcal{B}(x)$ satisfies \eqref{SSER} for all $x\ge 2$ and observed that
a positive answer would imply $H(x) \le A(x)$. 
Theorem \ref{corBrec} shows that $\mathcal{B}(x)$ does indeed satisfy \eqref{SSER} for all sufficiently large $x$. 
After removing multiples of members of $\mathcal{B}(x)$, the unsieved integers up to $x$ are exactly the members of 
$$
\mathcal{A}(x)=\{n\ge 1: F(n)\le x\}.
$$
Thus
$$
H(x) \le A(x) = \frac{(a+o(1)) x}{\log x}< \frac{1.53797 x}{\log x},
$$
by Theorem \ref{thm1}, for all sufficiently large $x$. 
The term $-\frac{\delta}{\log x}$ in Theorem \ref{corBrec} allows us to improve on this upper bound. 
For $\frac{1}{x} \le \tau \le 1$, define
$$
\mathcal{B}_{\tau}(x) = \mathcal{B}(\tau x) \cup \{p \in (\tau x,x]: p \text{ prime}\},
$$
as in Ruzsa \cite[Proof of the upper bound in Theorem I]{Ru}.
After removing multiples of members of $\mathcal{B}_{\tau}(x)$, the unsieved integers up to $x$ are exactly the members of 
$\mathcal{A}(\tau x)$. It follows that 
$$
H(x) \le H^*(x):=\min_{\tau \in \mathcal{T}(x)} A(\tau x),
$$
where 
$$
\mathcal{T}(x)=\Bigl\{\tau \ge 1/x: \sum_{n\in \mathcal{B}_\tau(x)} \frac{1}{n}\le 1\Bigr\}.
$$  
With Theorems \ref{thm1} and \ref{corBrec}, we will establish an estimate for $H^*(x)$.
\begin{theorem}\label{thmHasy}
Let $a$ and $\delta$ be as in Theorems \ref{thm1} and \ref{corBrec}.
For $x\ge 2$, we have
$$
H(x) \le H^*(x) = \frac{a e^{-\delta} x}{\log x}+ O\left(\frac{x}{(\log x)^{3/2}}\right) ,
$$
where $a e^{-\delta} \approx 0.878$. More precisely, $0.877992< a e^{-\delta} < 0.878171$. 
\end{theorem}

\begin{corollary}
For all sufficiently large $x$, we have
$$
H(x) < \frac{0.879 x}{\log x}. 
$$
\end{corollary}

For $2\le x \le 40$, we find that $H^*(x)-H(x) \in \{0,1\}$. 
It seems plausible that $H(x)\sim H^*(x)$ as $x\to \infty$.
\begin{question}
Is it true that
$
H(x) \sim  \frac{a e^{-\delta} x}{\log x} 
$
as $x\to \infty$ ?
\end{question}

Ruzsa \cite{Ru} also considered the more general problem of estimating 
$$
H(x,z) = \min_{\mathcal{S}} |\{n\le x: s \nmid n \text{ for all } s\in \mathcal{S}  \}|,
$$
where the minimum now is over all sets $ \mathcal{S}$ that satisfy
\begin{equation}\label{SSER2}
\sum_{n\in \mathcal{S}} \frac{1}{n} \le z, \quad 1\notin \mathcal{S}.
\end{equation}
Let 
$$
H^*(x,z):=\min_{\tau \in \mathcal{T}(x,z)} A(\tau x),
$$
where 
$$
\mathcal{T}(x,z)=\Bigl\{\tau \ge 1/x: \sum_{n\in \mathcal{B}_\tau(x)} \frac{1}{n}\le z\Bigr\}.
$$  
\begin{theorem}\label{thmHasy2}
Let $a$ and $\delta$ be as in Theorems \ref{thm1} and \ref{corBrec}.
Let $\mu >  -\delta$ be constant. For $x\ge 2$, we have 
$$
H\left(x, 1+\frac{\mu}{\log x}\right) \le H^*\left(x, 1+\frac{\mu}{\log x}\right) = \frac{a e^{-\delta -\mu} x}{\log x}+ O\left(\frac{x}{(\log x)^{3/2}}\right) .
$$
Let $Z>1$ be fixed. Uniformly for $1\le z \le Z$, $x\ge 2$,
$$
H\left(x,z\right) \asymp \frac{x^{\exp(1-z)}}{\log x}.
$$
\end{theorem}
The last estimate improves Ruzsa's result \cite[Theorem I]{Ru}
$$
\lim_{x\to \infty} \frac{\log H(x,z)}{\log x} = e^{1-z}.
$$

\subsection{Integers with dense divisors}

The significance of $F(n)$ to the distribution of divisors comes from Tenenbaum's identity \cite[Lemma 2.2]{Ten86}
$$
\frac{F(n)}{n} = \max_{1\le i <k} \frac{d_{i+1}}{d_i},
$$
where $1=d_1<d_2 < \cdots < d_k = n$ is the increasing sequence of divisors of $n$.
Thus
$$
D(x,y) := |\{n\le x: F(n)/n \le y \}|
$$
counts the number of integers up to $x$ whose sequence of divisors grows by factors of at most $y$. 
It may appear that the asymptotic estimates for $D(x,y)$ in \cite{PDD} should suffice to derive Theorem \ref{thm1}, but we have not been able to do so.
Instead, we derive Theorem \ref{thm1} from new estimates for the more general 
$$
D(x,y,z) := |\{n\le x: F(n)/n\le y, \ P^-(n)>z\}|,
$$ 
which has been considered previously by Saias \cite{EDD2} and the author \cite{IDD1}.

In Section \ref{secD}, we state these new estimates for $D(x,y,z)$, which are proved in Sections \ref{secproofduv} through \ref{secproofT2}.
The main result here is Theorem \ref{AF}, which is an extension of \cite[Theorem 1.3]{PDD} (where $z=1$) and is derived with the same overall strategy.
The ideas in \cite{IDD3,PDD,CFAE} generalize nicely to include the additional parameter $z$ without too much added difficulty.
The proofs in Section \ref{secA} suggest that the parameters $y$ and $z$ are both needed when trying to obtain sharp estimates for $A(x)$ from those for $D(x,y,z)$. 
As an added bonus, the new information about $D(x,y,z)$ leads to estimates not only for $A(x)$, but for the more general
$$
A(x,y,z) := |\{n\le x: F(n)\le xy, P^-(n)>z\}|,
$$
which are needed for the proofs of Theorems \ref{corBrec}, \ref{Rcor}, \ref{thmHasy} and \ref{thmHasy2}. 
In Section \ref{secacomp} we give algorithms for computing the constant factors in the asymptotics for $D(x,y,z)$ and $A(x,y,z)$,
mirroring similar computations in \cite{CFAE} for the constant factor in the asymptotic for $D(x,y)$. 
We prove Theorem \ref{corBrec} in Section \ref{secproofcorBrec} and Theorem \ref{Rcor} in Section \ref{secRcor}.
Section \ref{secproofHasy} contains the proofs of Theorems \ref{thmHasy} and \ref{thmHasy2}.
In Sections \ref{secnumbeta} and \ref{secnummuq}, we give the details of the computations of the constants $\beta$ (needed for $\delta$) and 
$\mu_q$ (needed for $\kappa$).

\section{Estimates for $D(x,y,z)$ and $A(x,y,z)$}\label{secD}

Define
$$
u=\frac{\log x}{\log y}, \quad v=\frac{\log x}{\log z}, \quad r=\frac{u}{v}=\frac{\log z}{\log y}.
$$
In \cite[Theorem 1]{IDD1} we found that 
\begin{equation}\label{IDD1T1}
D(x,y,z) = \frac{x d(u,v)}{\log z} +\frac{y}{\log y}-\frac{z}{\log z} +O\left(\frac{x}{\log^2 z}\right), \quad (x \ge y \ge z \ge 3/2),
\end{equation}
where the function $d(u,v)$ is defined by a certain difference-differential equation \cite[Eq. 3]{IDD1}.
This definition of $d(u,v)$ is based on the Buchstab identity
\begin{equation}\label{Dxyzp}
D(x,y,z) = 1 + \sum_{z<p \le y} D(x/p,py,p-0),
\end{equation}
which follows from grouping the integers counted in $D(x,y,z)$ according to their smallest prime factor $p$. 
In \cite{IDD3,PDD} we used a different kind of functional equation (i.e. Lemma \ref{lem0} with $z=1$) to show that
\begin{equation}\label{PDDThm}
D(x,y) = x \eta(y) d(u) \left(1+O\left(\frac{1}{\log x}\right)\right) \quad (x\ge y\ge 2),
\end{equation}
where $\eta(y) =1 +O(1/\log y)$ and
 $d(u)$ is given by $d(u)=0$ for $u<0$ and 
\begin{equation}\label{dinteq}
d(u)= 1-\int_0^{\frac{u-1}{2}} \frac{d(t)}{t+1} \ \omega\left(\frac{u-t}{t+1}\right) \, \mathrm{d} t \qquad (u\ge  0).
\end{equation}
Here and throughout, $\omega(u)$ denotes Buchstab's function. 
Equation \eqref{dinteq} can be solved with Laplace transforms to obtain \cite[Theorem 1]{IDD3}
\begin{equation}\label{IDD3T1}
d(u) = \frac{C}{u+1} \left(1+O\left(\frac{1}{u^2}\right)\right), \quad  C=\frac{1}{1-e^{-\gamma}}=2.280291...
\end{equation}

The following theorems generalize these results to $d(u,v)$ and $D(x,y,z)$. 

\begin{theorem}\label{duv}
For $v\ge u \ge 1$ we have
\begin{equation*}
\begin{split}
d(u,v) &= e^{-\gamma} (1-u/v) \, d(u) \left(1+ O\left(\frac{1}{u v}\right)\right)  =
\frac{C e^{-\gamma}  (1-u/v)}{u+1} 
\left(1+ O\left(\frac{1}{u^2}\right)\right) .
\end{split}
\end{equation*}
\end{theorem}

We will use the following notation throughout. Let 
$$
\Pi(z) := \prod_{p\le z}\left(1-\frac{1}{p}\right),\qquad \Sigma(z) := \sum_{p\le z} \frac{\log p}{p-1},
$$
and
\begin{equation}\label{etadef}
 \mathcal{E}(z):= \sup_{t\ge z} \left|\Sigma(t)+\gamma-\log t \right| \ll_\varepsilon \exp\left\{-(\log z)^{3/5-\varepsilon}\right\}, 
\end{equation}
for any fixed $\varepsilon >0$, by the prime number theorem.

\begin{theorem}\label{AF}
For $x\ge y \ge z \ge 3/2$, 
\begin{equation}\label{thmAFeq1}
D(x,y,z)= 
x d(u,v)e^\gamma\Pi(z)
 + \frac{x \beta_{y,z}}{\log xy}
+\frac{y}{\log y} - \frac{z}{\log z} 
+ O\left(\frac{x \log y}{\log^2 x \log z}\right),
\end{equation} 
where $\beta_{y,z} = c_{y,z}- C\Pi(z) \log (y/z)$ and $c_{y,z}$ is as in Theorem \ref{T2}. We have
\begin{equation}\label{thmAFeq2}
 |\beta_{y,z}| \le C \Pi(z) (\mathcal{E}(z)+\mathcal{E}(y)).
\end{equation}
\end{theorem}

Note that Theorem \ref{AF} implies \eqref{IDD1T1}.
A little exercise shows that Theorems \ref{duv} and \ref{AF} imply \eqref{PDDThm}.
 Let $\chi_{y,z}(n)$ be the characteristic function of the set
\begin{equation}\label{Dsetdef}
 \mathcal{D}_{y,z}:=\{n\in \mathbb{N}: F(n)/n \le y, P^-(n)>z \}.
\end{equation}
Let $\pi(y,z)$ be the number of primes $p$ satisfying $z<p\le y$, that is
$$
\pi(y,z) := \sum_{z<p\le y} 1 .$$

\begin{theorem}\label{T2}
For $x\ge y \ge z \ge 3/2$ and $y\ll \frac{x}{\log x}$ we have
\begin{equation}\label{Dasymp}
D(x,y,z)=1+ \frac{c_{y,z} x}{\log xy}\left\{1+O\left(\frac{1}{\log x }+\frac{\log^2 y}{\log^2 x}\right)\right\},
\end{equation}
and, for $x\ge y\ge z \ge 3/2$,
\begin{equation}\label{Dasymp2}
D(x,y,z)=1+ \pi(y,z)+\frac{c_{y,z} x}{\log xy}\left\{1+O\left(\frac{1}{\log x }+\frac{\log^2 y}{\log^2 x}\right)\right\},
\end{equation}
where $c_{y,z}$ is given by
\begin{equation}\label{tdensecon}
c_{y,z} = C\sum_{n\ge 1}\frac{\chi_{y,z}(n)}{n} \bigl(\Sigma(yn)-\Sigma(z) -\log n\bigr) \Pi(yn).
\end{equation}
For $y\ge z \ge 1$,
\begin{align}
c_{y,z} & = C \Pi(z)(\log y -\gamma - \Sigma(z) +\zeta_{y,z}) \quad  & |\zeta_{y,z}|  \le \mathcal{E}(y), \label{cyz1} \\ 
c_{y,z} &= C \Pi(z)(\log (y/z) +\psi_{y,z})                           & |\psi_{y,z}|  \le \mathcal{E}(z)+ \mathcal{E}(y),  \label{cyz1b} \\
c_{y,z} &= C(\Pi(z)-\Pi(y))\left(\log y+\xi_{y,z}\right)       & |\xi_{y,z}| \le  \mathcal{E}(z)+\mathcal{E}(y),\label{cyz2} \\ 
c_{y,z} & \asymp(\Pi(z)-\Pi(y)) \log y. \label{cyz3}
\end{align}
\end{theorem}

\begin{table}[h]
\caption{Values of the factor $c_{y,z}$.}\label{tab1}
  \begin{tabular}{ | l | l | l | l | l | }
    \hline
    &        $z=1$         & $z=2$      & $z=3$        & $z=5$ \\ \hline
    $y=2$ &  1.2248... &   0            &    0           & 0\\ \hline
    $y=3$ &  2.0554... &   0.4315... &      0         & 0\\ \hline
    $y=4$ &  2.4496... &   0.5242... &      0         & 0\\ \hline
    $y=5$ &  2.9541... &   0.8402... & 0.2351... & 0\\ \hline
    $y=6$ &  3.2477...   &   0.9263... & 0.2574... & 0\\ \hline
    $y=7$ &  3.6441...   &   1.1573... & 0.4321... & 0.1544... \\ \hline
     \end{tabular}
\end{table}

The algorithm for calculating the values in Table \ref{tab1} is described in Section \ref{secacomp}.
In Section \ref{secA}, we derive the following two theorems from Theorems \ref{AF} and \ref{T2}. 

\begin{theorem}\label{thmAD}
For $x\ge y \ge z \ge 3/2$ we have
\begin{equation}\label{thmADeq}
A(x,y,z)=\pi(\sqrt{xy},z)+ x d(u,v) e^{\gamma} \Pi(z)+ \frac{\tau_{y,z} x}{\log xy}+ O\left(\frac{x \log y}{\log^2 x \log z}\right),
\end{equation}
where $\tau_{y,z} = a_{y,z}-C\Pi(z) \log (y/z)$ and $a_{y,z}$ is as in Theorem \ref{thmA}. We have
\begin{equation}\label{taueq2}
\tau_{y,z} = C\Pi(z) \left(1+\eta_{y,z}\right), \qquad  |\eta_{y,z}| \le \mathcal{E}(z)+\mathcal{E}(y).
\end{equation}
\end{theorem}

\begin{theorem}\label{thmA}
For $y\ge 1$,  $z \ge 1$, $x\ge \max(2,y,z)$,  we have
\begin{equation}\label{Aasymp1}
A(x,y,z) =1+\pi(\sqrt{xy},z)+ \frac{a_{y,z} x}{\log xy}\left(1+O\left(\frac{1}{\log xy } + \frac{\log^2 2zy}{\log^2 xy}\right)\right).
\end{equation}
If $y\ge z \ge 1$,
\begin{equation}\label{adef2}
 a_{y,z}=C\Pi(z) \left( 1 -\gamma +\log y -\Sigma(z)\right)+  C\int_1^\infty A(t,y,z) g(yt) \frac{dt}{t^2}  ,
\end{equation}
where
$ g(t) =(\Sigma(t) +\gamma - \log t  ) \Pi(t) .$
If $z \ge y\ge 1$, then $a_{y,z} =\frac{y}{z} a_{z,z}$.

For $y\ge z \ge 1$,
\begin{align}
a_{y,z} & =C\Pi(z) \left(1-\gamma + \log y  -\Sigma(z)+\delta_{y,z}\right) \qquad  & |\delta_{y,z}| \le \mathcal{E}(y),\label{ateq1} \\
a_{y,z} & = C\Pi(z) \left(1+ \log(y/z)+\eta_{y,z}\right) \qquad   & |\eta_{y,z}| \le \mathcal{E}(z)+\mathcal{E}(y).\label{ateq2}
\end{align}
For $y\ge 1$, $z\ge 1$, 
\begin{equation}\label{ateq4}
 a_{y,z} \asymp  \Pi(z) \log(1+y/z).
\end{equation}
\end{theorem}

\begin{table}[h] \caption{Values of the factor $a_{y,z}$.}\label{tab2}
  \begin{tabular}{ | l | l | l | l | l | }
    \hline
              & $z=1$           &  $z=2$        & $z=3$        & $z=5$ \\ \hline
    $y=1$ &  1.5379...   &   0.4178...  &  0.1583...  &   0.0831... \\ \hline
    $y=2$ &  3.0759...   &   0.8357...  &  0.3167...  &   0.1662... \\ \hline
    $y=3$ &  3.9184...   &   1.2535...  &  0.4751...  &   0.2493... \\ \hline
    $y=4$ &  4.4804...   &   1.5153...  &  0.6335...  &   0.3325... \\ \hline
    $y=5$ &  4.9557...   &   1.7525...  &  0.7918...  &   0.4156... \\ \hline
    $y=6$ &  5.3297...   &   1.9280...  &  0.9015...  &   0.4987... \\ \hline
     \end{tabular}
\end{table}

Theorem \ref{thm1} follows from Theorem \ref{thmA} with $y=z=1$. Theorems \ref{thmAD} and \ref{thmA} sharpen earlier results by
Saias \cite[Lem. 5]{AEDD} and the author \cite[Thm. 1]{IDD1}. The algorithm for calculating the values in Table \ref{tab2} is 
described in Section \ref{secacomp}.

\section{Proof of Theorem \ref{duv}}\label{secproofduv}
Let
$$\Phi(x,z) = |\{n\le x: P^-(n)>z \}|.$$
Note that for $n \ge 2$ with  prime factorization $n=p_1 \cdots p_k$, where  $p_1\le p_2\le \ldots \le p_k$, we have
$$
\frac{F(n)}{n}\le y \quad \Leftrightarrow \quad p_{j+1} \le y \prod_{1\le i\le j} p_i \quad (0\le j<k).  
$$

\begin{lemma}\label{lem0}
For $x\ge 1$, $y\ge z \ge 1$ we have 
$$\Phi(x,z) = \sum_{n\le x}\chi_{y,z}(n) \Phi(x/n, yn).$$
\end{lemma}

\begin{proof}
Every $m$ counted in $\Phi(x,z)$ factors uniquely as $m=nr$, where $n\in \mathcal{D}_{y,z}$  and $P^-(r)>ny$.
\end{proof}

\begin{lemma}\label{lem2}
For $x\ge 1$,  $y\ge z \ge 1$ we have 
\begin{equation*}
D(x,y,z)-1=\Phi(x,z)-\Phi(x,y)-\sum_{z< n\le \sqrt{x/y}}\chi_{y,z}(n) \bigl(\Phi(x/n, yn)-1\bigr)
\end{equation*}
\end{lemma}

\begin{proof}
This follows from Lemma \ref{lem0}, since $1\in \mathcal{D}_{y,z}$, $2\le n \le z \Rightarrow
n \notin \mathcal{D}_{y,z}$, and $n>\sqrt{x/y} \Rightarrow \Phi(x/n, yn)=1$.
\end{proof}

\begin{lemma}\label{Phi}
For $x\ge 1$, $z\ge 3/2$, we have
\begin{equation*}
\Phi(x,z) -1 =  x \Pi(z)
+\frac{x(\omega(v)-e^{-\gamma})}{\log z}-\left. \frac{z}{\log z}\right|_{x\ge z}
+ O\left(\frac{x e^{-v/3}}{\log^2 z}\right),
\end{equation*}
where $v=\frac{\log x}{\log z}$ and $\omega(v)$ is Buchstab's function. 
\end{lemma}

\begin{proof}
This is Lemma 6 of \cite{OMB}.
\end{proof}

We write $d_r(u):=d(u,u/r)$ for $0<r\le 1$ and $d_0(u):= e^{-\gamma} d(u)$.

\begin{lemma}\label{lem3}
For $u \ge 0$, $0<r\le 1$ we have
$$d_r(u)=\omega(u/r)-r\omega(u)-\int_{0}^{u} \frac{d_r(t)}{t+1}\, \omega\left(\frac{u-t}{t+1}\right) \mathrm{d}t .$$
\end{lemma}

\begin{proof}
When $u\le 1$, this follows from the initial conditions for $d(u,v)$ \cite[Eq. (4)]{IDD1}.
Assume $u>1$ and $0<r\le 1$ are fixed. 
In Lemma \ref{lem2}, we estimate $D(x,y,z)$ with \eqref{IDD1T1} and apply 
$$
\Phi(x,z)= \frac{x \omega(v) - z}{\log z} + O\left(\frac{x}{\log^2 z}\right), \quad (x\ge z \ge 3/2),
$$
which follows from Lemma \ref{Phi}. 
In the sum of Lemma \ref{lem2}, first estimate $\Phi(x/n, yn)$, 
then use partial summation and estimate $D(x,y,z)$ with \eqref{IDD1T1}.
\end{proof}

Note that as $r\to 0$, the equation in Lemma \ref{lem3} turns into the equation for
$d(u)$ in \eqref{dinteq}, multiplied by $e^{-\gamma}$. 

In order to calculate Laplace transforms, we make the following change of variables. Let
\begin{equation}\label{Frdef}
 G_r(y)=  d_r(e^y-1)  \quad (0\le r \le 1),
\end{equation}
\begin{equation}\label{Omegardef}
\Omega_r(y)=\omega\left(\frac{e^y-1}{r}\right) \quad (0< r \le 1), \qquad \Omega_0(y)=e^{-\gamma}.
\end{equation}
\begin{lemma}\label{lem4}
For $0\le  r \le 1$ and $\mathrm{Re}(s) >0$ we have
$$ \widehat{G}_r(s) =\frac{\widehat{\Omega}_r(s)-r \widehat{\Omega}_1(s)}{1+\widehat{\Omega}_1(s)}.$$
\end{lemma}

\begin{proof}
In Lemma \ref{lem3}, changing variables $u=e^\lambda-1$, $t=e^\mu-1$, leads to 
$$
G_r(\lambda)=\Omega_r(\lambda)-r\Omega_1(\lambda)-\int_0^\lambda G_r(\mu) \Omega_1(\lambda-\mu) d\mu.
$$
Calculating the Laplace transform, we have, for $\mathrm{Re}(s)  >0$,
$$
 \widehat{G}_r(s) = \widehat{\Omega}_r(s)-r \widehat{\Omega}_1(s) -\widehat{G}_r(s) \widehat{\Omega}_1(s),
$$
which yields the stated result.
\end{proof}

\begin{proof}[Proof of Theorem \ref{duv}]
Let $r=u/v$, $\lambda=\log(u+1)$, and 
$$
E_r(\lambda):=\Omega_r(\lambda)-e^{-\gamma}-r(\Omega_1(\lambda)-e^{-\gamma}).
$$
Let $G(\lambda):=d(e^\lambda -1)$, where $d(u)$ is as in \eqref{dinteq}.
As in the proof of Lemma \ref{lem4}, equation \eqref{dinteq} leads to
\begin{equation}\label{FLap}
 \widehat{G}(s)= \frac{1}{s(1+\widehat{\Omega}_1(s))},
\end{equation}
for $\mathrm{Re}(s)>0$. With Lemma \ref{lem4}, we find that
$$
 \widehat{G}_r(s) = e^{-\gamma}(1-r) \widehat{G}(s) + \widehat{E}_r(s) s \widehat{G}(s).
$$
We have
\begin{equation}\label{sFseq}
 s\widehat{G}(s) = \widehat{G'}(s)+G(0)=\widehat{G'}(s)+1
\end{equation}
and
\begin{equation}\label{Fplameq}
 G'(\lambda)=d'(e^\lambda-1) e^\lambda =  -C e^{-\lambda}+O\left(e^{-3\lambda}\right),
\end{equation}
by Corollary 5 of \cite{IDD3}.
Thus
\begin{equation}\label{Freq}
G_r(\lambda) =  e^{-\gamma}(1-r)  G(\lambda) + E_r(\lambda) + \int_0^\lambda E_r(\tau) \left(-C e^{-(\lambda-\tau)} +O( e^{-3(\lambda-\tau)}) \right) d\tau.
\end{equation}
Considering the two cases $0\le r\le 1/2$ and $1/2<r \le 1$, we find that for any fixed $A>0$ we have
$$
E_r(\tau) \ll_A r (1-r) e^{-A\tau}.
$$
Moreover, 
$$
\int_0^\infty E_r(\tau) e^\tau d\tau = \int_0^\infty \left[ (\omega(u/r)-e^{-\gamma})-r(\omega(u)-e^{-\gamma})\right] du = 0.
$$
This shows that the integral in \eqref{Freq} amounts to $\ll r(1-r) e^{-3\lambda} $. Thus \eqref{Freq} simplifies to
$$
G_r(\lambda) =  e^{-\gamma}(1-r)  G(\lambda) +O(r(1-r) e^{-3\lambda}),
$$
that is
$$
d_r(u) = e^{-\gamma}(1-r) d(u) + O(r(1-r)u^{-3}).
$$
The first estimate in Theorem \ref{duv} now follows, since $d(u) \asymp 1/u$ for $u\ge 1$. 
The second estimate follows from the first and \eqref{IDD3T1}.
\end{proof}

\section{Proof of Theorem \ref{AF}}

Let 
$$
\delta(x)=
\begin{cases} 1 &\mbox{if } x \ge 0 \\ 
0 & \mbox{if } x<0. \end{cases} 
$$

\begin{lemma}\label{UB}
For $x\ge 1$, $y\ge z \ge 3/2$,
$$ D(x,y,z)-1 \ll \frac{x \log y}{\log xy \log z}.$$
\end{lemma}

\begin{proof}
If $x\ge z$, this follows from \cite[Lemma 5]{AEDD}. If $x<z$ then $D(x,y,z)=1$.
\end{proof}

Define
\begin{equation}\label{Edef}
E(x,y,z):=\frac{x \log y}{(\log xy)^2 \log z}.
\end{equation}

\begin{lemma}\label{A1}
For $x\ge 1$, $y\ge z \ge 3/2$,
\begin{multline*}
D(x,y,z)-1 = \Phi(x,z)-\Phi(x,y) +O\left(E(x,y,z)\right) \\
- x\sum_{z<n\le \sqrt{x/y}}
\frac{\chi_{y,z}(n)}{n}\left\{\Pi(yn)+ \frac{1}{\log yn}\left( \omega\left(\frac{\log x/n}{\log yn}\right) -e^{-\gamma} \right) \right\} .
\end{multline*}
\end{lemma}

\begin{proof}
We use Lemma \ref{lem2}. The case $x\le yz^2$ is trivial, since the sum in Lemma \ref{lem2} vanishes. If $x>yz^2$, we use Lemma \ref{Phi}
to estimate each occurrence of $\Phi(x/n,ny)-1$. 
The term $\frac{z\,\delta(x-z)}{\log z}$ contributes
$$
\sum_{z<n\le \sqrt{x/y}} \chi_{y,z}(n) \frac{ny}{\log ny} \ll \frac{\sqrt{xy}}{\log \sqrt{xy}} (D(\sqrt{x/y},y,z)-1) \ll E(x,y,z),
$$
by Lemma \ref{UB}.
For the contribution of the error term in Lemma \ref{Phi}, we can split the interval $(z,\sqrt{x/y}]$ by powers of $2$ and use Lemma \ref{UB}. 
This shows that the error term in Lemma \ref{Phi} contributes
\begin{equation*}
\begin{split}
& \ll  \sum_{z<n\le \sqrt{x/y}} \chi_{y,z}(n)  \frac{x}{n (\log ny)^2} \exp\left(-\frac{\log xy}{3\log ny}\right)\\
& \ll \sum_{z<n\le \sqrt{x/y}} \frac{x \log y}{n (\log ny)^3 \log z} \exp\left(-\frac{\log xy}{6\log ny}\right)
\ll E(x,y,z).
\end{split}
\end{equation*}
\end{proof}

\begin{lemma}\label{A2}
For $y\ge z \ge 3/2$, we have
$\displaystyle \Pi(z) = 
\sum_{n\ge 1} \frac{\chi_{y,z}(n)}{n}\Pi(yn).
$
\end{lemma}

\begin{proof}
Fix $y\ge z \ge 3/2$. In Lemma \ref{A1}, use the estimate $\omega(t)-e^{-\gamma} \ll e^{-t}$. 
The contribution from this error term can be estimated as in the proof of Lemma \ref{A1}.
Use Lemma \ref{UB} for $D(x,y,z)-1$, divide the resulting equation by $x$ and let $x\to \infty$.
\end{proof}

\begin{lemma}\label{A3}
For $x \ge 1$, $y\ge z \ge 3/2$, we have
\begin{multline*}
D(x,y,z)-1 = 
\frac{x(\omega(v)-e^{-\gamma})}{\log z}-\frac{x(\omega(u)-e^{-\gamma})}{\log y}\\
-x \sum_{n>z}
\frac{\chi_{y,z}(n)}{n \log yn} 
\left(  \omega\left(\frac{\log x/n}{\log yn}\right)-e^{-\gamma}  \right) \\
+ \frac{y\, \delta(x-y)}{\log y}-\frac{z\, \delta(x-z)}{\log z}
+O\left(E(x,y,z)
+\frac{x e^{-v/3}}{\log^2 z} \right).
\end{multline*}
\end{lemma}

\begin{proof}
In Lemma \ref{A1}, we use Mertens' theorem to extend the sum to all $n>z$. This changes the sum by $\ll E(x,y,z)$, since $\omega(t)=0$ when $t<1$.
Approximate $\Phi(x,z)$ and $\Phi(x,y)$  in Lemma \ref{A1} by Lemma \ref{Phi}. By Lemma \ref{A2}, all terms with Euler products cancel. 
The result now follows since $x e^{-u/3}/\log^2 y \ll E(x,y,z)$. 
\end{proof}

\begin{lemma}\label{A4}
For $x \ge 1$, $y\ge z \ge 3/2$, we have
\begin{multline}\label{lemA4eq}
D(x,y,z)-1 =
\frac{x \omega(v)}{\log z}-\frac{x \omega(u)}{\log y}\\
+x\alpha_{y,z}  - x \int_1^x \frac{D(t,y,z)-1}{t^2 \log t y } 
\, \omega\left(\frac{\log x/t}{\log t y}\right) \, \mathrm{d}t\\
+ \frac{y\, \delta(x-y)}{\log y}-\frac{z\, \delta(x-z)}{\log z}
+O\left(E(x,y,z) +\frac{x e^{-v/3}}{\log^2 z}\right),
\end{multline}
where 
$$ \alpha_{y,z} :=\frac{e^{-\gamma}}{\log y}- \frac{e^{-\gamma}}{\log z}+e^{-\gamma} \int_z^\infty \frac{D(t,y,z)-1}{t^2 \log t y }\, \mathrm{d}t .$$
We have
\begin{equation}\label{alphaest}
 \alpha_{y,z} =  \Pi(z) - \frac{e^{-\gamma}}{\log z} 
+ O\left(\frac{1}{\log y \log z}\right).
\end{equation}
\end{lemma}

\begin{proof}
Use partial summation on the sum in Lemma \ref{A3} to turn the sum over $n>z$ into an integral $\int_z^\infty$. 
The two new error terms arising during partial summation are found to be $\ll E(x,y,z)$, with the help of Lemma \ref{UB}. 
After separating the contribution from $e^{-\gamma}$ to the integral, we can change the limits in the remaining integral from $\int_z^\infty$ to 
$\int_1^x$, since $D(t,y,z)-1=0$ when $1\le t \le z$ and $\omega(t)=0$ for $t\le 0$.

To see that \eqref{alphaest} holds, put $x=y$ in \eqref{lemA4eq} and use Lemma \ref{Phi} to estimate $D(x,x,z)-1=\Phi(x,z)-1$. 
The integral in \eqref{lemA4eq} vanishes when $x=y$, since $\omega(t)=0$ for $t<1$. 
\end{proof}

\begin{proof}[Proof of Theorem \ref{AF}]
For $x\ge 1$, $t \ge 1$, define $\lambda\ge 0$ and $ \tau \ge 0$ by
$$ x=y^{e^\lambda -1}, \quad t = y^{e^\tau -1}.$$
Let 
$$ H_{y,z}(\lambda) := \frac{D( y^{e^\lambda -1}, y,z)-1}{y^{e^\lambda -1}} 
=\frac{D(x,y,z)-1}{x} \ll \frac{\log y}{\log xy \log z},$$
by Lemma \ref{UB}. 
Let $G_r(\lambda)$  and $ \Omega_r(\lambda)$ be as in \eqref{Frdef} and \eqref{Omegardef}.
Lemma \ref{A4} shows that for $\lambda\ge 0$ we have
$$
H_{y,z}(\lambda) =\frac{\Omega_r(\lambda)-r\Omega_1(\lambda) }{\log z}
+\alpha_{y,z}
-\int_0^\lambda H_{y,z}(\tau) \Omega_1(\lambda-\tau) \, \mathrm{d} \tau
+R_{y,z}(\lambda),
$$
where
$$
R_{y,z}(\lambda) := 
 \frac{ \delta\bigl(\lambda - \log 2\bigr)}{y^{e^\lambda -2}\log y}
-\frac{\delta\bigl(\lambda - \log(1+r)\bigr)}{y^{e^\lambda -(1+r)}\log z}
+O\left(\frac{e^{-2\lambda}}{\log y \log z}
+\frac{ e^{ -(e^\lambda -1)/(3r)}}{\log^2 z}  \right).
$$
Calculating Laplace transforms, we obtain, for $\mathrm{Re}(s)>0$,
\begin{equation*}
\widehat{H}_{y,z}(s) = \frac{\widehat{\Omega}_r(s)-r \widehat{\Omega}_1(s)}{\log z}
+\frac{\alpha_{y,z}}{s}-\widehat{H}_{y,z}(s) \widehat{\Omega}_1(s) + \widehat{R}_{y ,z}(s).
\end{equation*}
By Lemma \ref{lem4} and \eqref{FLap},
\begin{equation*}
\widehat{H}_{y,z}(s) =
\frac{\widehat{G}_r(s)}{\log z}
+\alpha_{y,z}\widehat{G}(s)
  + \widehat{R}_{y,z}(s) s\widehat{G}(s) .
\end{equation*}
From \eqref{sFseq} and \eqref{Fplameq}, we obtain
\begin{equation*}
\begin{split}
H_{y,z}(\lambda) 
=   \frac{G_r(\lambda)}{\log z}  +\alpha_{y,z} G(\lambda)+ R_{y,z}   +\int_0^\lambda R_{y,z}(\tau) \Bigl(-Ce^{-(\lambda-\tau)}+O\bigl(e^{-3(\lambda-\tau)} \bigr)\Bigr) \,\mathrm{d}\tau.
\end{split} 
\end{equation*}
In the remainder of this proof we may assume $x\ge y \ge z \ge 3/2$.
The last integral equals
\begin{equation*}
\begin{split}
&  e^{-\lambda} \int_0^\infty -C  R_{y,z}(\tau) e^\tau  \,\mathrm{d}\tau
+O\left(\int_\lambda^\infty | R_{y,z}(\tau ) | e^{\tau -\lambda}\,\mathrm{d}\tau
+  \int_0^\lambda |R_{y,z}(\tau)|  \, e^{3(\tau - \lambda)} \,\mathrm{d}\tau \right) \\
& =:  e^{-\lambda} \kappa_{y,z} \  + \ O\left(\frac{e^{-\lambda}}{\log x \log z}\right),
\end{split}
\end{equation*}
where
$$\kappa_{y,z} \ll \frac{1}{\log y \log z} .$$
Hence
\begin{equation*}
\begin{split}
D(x,y,z) = & \frac{ x \, d(u,v)}{\log z}+ x\alpha_{y,z} d(u)
 + \frac{x \, \kappa_{y,z}}{u+1}
+ \frac{y}{\log y}-\frac{z}{\log z}
+ O\left(E(x,y,z)\right).
\end{split}
\end{equation*}
From \eqref{alphaest} we have 
$$ \alpha_{y,z} =  \Pi(z) - \frac{e^{-\gamma}}{\log z} 
+ \rho_{y,z},
\quad  \rho_{y,z}\ll \frac{1}{\log y \log z}, $$
and, by Mertens' theorem,
$$
\varepsilon_z := \alpha_{y,z}-\rho_{y,z} =  \Pi(z)  - \frac{e^{-\gamma}}{\log z}  \ll \frac{1}{\log^2 z}.
$$ 
With the estimate \eqref{IDD3T1}, we can write
$$
x \alpha_{y,z} d(u) =x(\varepsilon_z + \rho_{y,z}) d(u) =x\varepsilon_z d(u)  + \frac{x C \rho_{y,z}}{u+1} +O(E(x,y,z)).
$$
Thus
\begin{equation*}
\begin{split}
D(x,y,z) = & \frac{ x \, d(u,v)}{\log z}+ x\varepsilon_z d(u)
 + \frac{x \, \mu_{y,z}}{u+1}
+ \frac{y}{\log y}-\frac{z}{\log z}
+ O\left(E(x,y,z)\right),
\end{split}
\end{equation*}
where
$$
\mu_{y,z}:=\kappa_{y,z}+C \rho_{y,z} \ll \frac{1}{\log y \log z}.
$$
If $v\ll 1$, that is $\log z \gg \log x$, then the result follows from Mertens' theorem. 
Assume $v$ is sufficiently large. Then the first estimate in Theorem \ref{duv} shows that, with $r=u/v$,
\begin{equation*}
\begin{split}
x\varepsilon_z d(u) &= x\varepsilon_z d(u,v) e^\gamma \left(1+\frac{r}{1-r}\right) \left(1+O\left(\frac{1}{uv}\right)\right) \\
& = x\varepsilon_z d(u,v) e^\gamma +\frac{r  x\varepsilon_z d(u,v) e^\gamma }{1-r} +O\left(\frac{x\varepsilon_z d(u,v) }{(1-r)uv}\right).
\end{split}
\end{equation*}
Since $d(u,v)\ll \frac{1-r}{u+1}$, the last error term is $\ll E(x,y,z)$. 
The second estimate in Theorem \ref{duv} yields
$$
\frac{r  x\varepsilon_z d(u,v) e^\gamma }{1-r} =x \frac{\varepsilon_z C r}{u+1} +O(E(x,y,z)).
$$ 
The estimate \eqref{thmAFeq1} now follows with 
$$
\beta_{y,z} := (\log y)(\mu_{y,z} + \varepsilon_z C r) \ll \frac{1}{ \log z}.
$$
It remains to prove the stronger upper bound \eqref{thmAFeq2}. For fixed $y\ge z \ge 3/2$, \eqref{thmAFeq1} 
and Theorem \ref{duv} yield
\begin{equation}\label{Dsmally}
D(x,y,z) = \frac{c_{y,z} x}{\log xy} + O_{y,z}\left(\frac{x}{\log^2 x}\right),
\end{equation}
where
\begin{equation}\label{cyzeq}
c_{y,z} = C \Pi(z) \log(y/z)+  \beta_{y,z}.
\end{equation}
With \eqref{Dsmally} and the method in \cite{CFAE}, 
just like in the derivation of the formula for $c_y=c_{y,1}$ in \cite{CFAE}, we derive \eqref{tdensecon}:
As a first step, \cite[Lemma 1]{CFAE} is replaced by the more general identity
$$
1 = \sum_{n\ge 1} \frac{\chi_{y,z}(n)}{n^s} \prod_{z<p\le ny} \left(1-\frac{1}{p^s}\right) \qquad (\mathrm{Re}(s)>1),
$$
the proof of which is analogous to that of \cite[Lemma 1]{CFAE}. For the rest of the proof of \eqref{tdensecon}, 
use this modified version of Lemma 1 and follow the
proofs of Lemmas 2 through 4 of \cite{CFAE}.

With \eqref{etadef} and Lemma \ref{A2}, we obtain \eqref{cyz1} and \eqref{cyz1b}.
Comparing \eqref{cyz1b} with \eqref{cyzeq} establishes \eqref{thmAFeq2}.
\end{proof}

\section{Proof of Theorem \ref{T2}}\label{secproofT2}

For $y\ge z \ge 1$, let $\Lambda(y,z)$ be the natural density of integers whose smallest prime factor is in the interval $(z,y]$, that is
$$\Lambda(y,z) :=\Pi(z)-\Pi(y)
= \sum_{z<p\le y} \frac{\Pi(p-1) }{p} \asymp \sum_{z<p\le y} \frac{1}{p\log p}.
$$

We first establish the following corollary to Theorem \ref{AF}.
\begin{corollary}\label{corlambda2}
For $x\ge y \ge z \ge 3/2$ we have
\begin{equation}\label{corlam2eq}
D(x,y,z) = 1+\pi(y,z)+\frac{C x  \Lambda(y,z)\bigl(\log y +\xi_{y,z}\bigr)}{\log xy}
\left(1+O\left(\frac{1}{\log x} + \frac{\log^2 y}{\log^2 x}  \right)\right),
\end{equation}
where
$$ |\xi_{y,z}| \le \mathcal{E}(z)+\mathcal{E}(y), \quad \log y +\xi_{y,z} \asymp \log y.$$
\end{corollary}
\begin{proof}
If the interval $(z,y]$ does not contain a prime, then $D(x,y,z)=1$ and $\pi(y,z)=\Lambda(y,z)=0$, so \eqref{corlam2eq} holds.
Assume $\pi(y,z)\ge 1$, so that $\Lambda(y,z)>0$. 

We first consider the case $y\le x^{1/4}$. 
As in \eqref{Dxyzp}, we write
$$D(x,y,z)=1+\sum_{z<p\le y} D(x/p,py,p-0)$$
 and apply Theorem \ref{AF} to estimate each occurrence of $D(x/p,py,p-0)$. 
Note that $x/p \ge py$, since $p\le y$ and $y\le x^{1/4}$. The contribution from $y/\log y - z/\log z$ in Theorem \ref{AF} is 
$$
< \sum_{z<p\le y} \frac{py}{\log py} \le \sum_{z<p\le y} \frac{p^2y}{p \log p} \le   \sum_{z<p\le y} \frac{y^3}{p \log p} \ll x^{3/4}\Lambda(y,z).
$$
The contribution from the term $x \beta_{y,z}/\log xy$ in Theorem \ref{AF} is
$$
= \frac{x}{\log xy}  \sum_{z<p\le y} \frac{\beta_{py,p-0}}{p} =\frac{x}{\log xy}  \cdot C \Lambda(y,z) \xi_{y,z},
$$
where we define
$$
\xi_{y,z} :=\frac{1}{C \Lambda(y,z)}  \sum_{z<p\le y} \frac{\beta_{py,p-0}}{p}.
$$
The estimate \eqref{thmAFeq2} yields 
$|\xi_{y,z}| \le \mathcal{E}(z)+\mathcal{E}(y)$. With the second estimate for $d(u,v)$ in Theorem \ref{duv}, 
the contribution from the term $x d(u,v) e^\gamma \Pi(z)$ in Theorem \ref{AF} is
$$
\frac{Cx  \Lambda(y,z) \log y}{\log xy}
\left(1+O\left(\frac{\log^2 y}{\log^2 x}\right)\right).
$$
The contribution from the error term in Theorem \ref{AF} is 
$$\ll \frac{x \Lambda(y,z)\log y }{ \log^2 x}. $$ 
We have
$$
\frac{x \log y \Lambda(y,z)}{(\log x)^2} \gg \frac{x \log y}{\log^2 x}\cdot \frac{\pi(y,z)}{y \log y} \gg \pi(y,z) \gg 1+\pi(y,z).
$$
The result now follows if $y>y_0$, since $\log y+\xi_{y,z} = \log y +O(1) \asymp \log y$.  
If $y\le y_0$, we also have $\log y+\xi_{y,z} \asymp \log y$, since the preceding shows that
$$
D(x,y,z) = O\left(\frac{x}{\log^2 x}\right) + \frac{C x  \Lambda(y,z)\bigl(\log y+\xi_{y,z}\bigr)}{\log xy},
$$
and $D(x,y,z) \asymp_{y,z} \frac{x}{\log x}$ when $\frac{3}{2}\le z\le y \le y_0$ and $\pi(y,z)\ge 1$, by \cite[Lemma 3]{PPN}.
This concludes the proof of the case $y\le x^{1/4}$.

If $y>x^{1/4}$, we use the estimate
$$
D(x,y,z)-1 \le \Phi(x,z)-1 \ll \frac{x}{\log z} \qquad (x\ge 1, y\ge 1, z>1),
$$
which imples
$$
D(x,y,z)-1-\pi(y,z)=\sum_{z<p\le y} (D(x/p,py,p-0)-1) \ll \sum_{z<p\le y} \frac{x}{p\log p} \asymp x \Lambda(y,z).
$$
The result now follows since $\log x \asymp \log y$ in this case. 
\end{proof}

When  $y \ll x/\log x$, the term $\pi(y,z)$ is absorbed by the error terms. 

\begin{corollary}\label{corlambda}
For $x\ge y \ge z \ge  3/2$ and $y \ll \frac{x}{\log x}$ we have
\begin{equation*}
D(x,y,z) =1+\frac{C x \Lambda(y,z)\bigl(\log y +\xi_{y,z}\bigr)}{\log xy}
\left(1+O\left(\frac{1}{\log x}+\frac{\log^2 y}{\log^2 x}  \right)\right),
\end{equation*}
where
$$ |\xi_{y,z}| \le \mathcal{E}(z)+\mathcal{E}(y), \quad \log y+\xi_{y,z} \asymp \log y.$$
\end{corollary}

\begin{proof}
When $y\le x^{1/4}$, the proof of Corollary \ref{corlambda2} shows that $\pi(y,z)$ is absorbed by the first error term.
If $x^{1/4} < y \ll \frac{x}{\log x}$, we have
$$
\pi(y,z) \ll x \frac{\pi(y,z)}{y\log y} \le x \sum_{z<p\le y} \frac{1}{p\log p} \ll x \Lambda(y,z)
$$
and the result follows from Corollary \ref{corlambda2}.
\end{proof}

\begin{proof}[Proof of Theorem \ref{T2}]
For $y\ge z \ge 3/2$ fixed and $x \to \infty$, comparing \eqref{Dsmally} and Corollary \ref{corlambda} shows that
$$
c_{y,z}= C \Lambda(y,z)(\log y +\xi_{y,z}). 
$$
The first estimate in Theorem \ref{T2} is Corollary \ref{corlambda} and the second one is Corollary \ref{corlambda2}.
The formula for $c_{y,z}$ in \eqref{tdensecon} and the estimates \eqref{cyz1} and \eqref{cyz1b} were established in the proof of Theorem \ref{AF}, while 
\eqref{cyz2} and \eqref{cyz3} follow from the estimates for $\xi_{y,z}$ in Corollary \ref{corlambda}.
In Section \ref{secacomp} we describe the calculation of the numerical values of $c_{y,z}$ in Table \ref{tab1}.
\end{proof}

\section{Proof of Theorems \ref{thmAD} and \ref{thmA}}\label{secA}

Recall that $\chi_{y,z}(n)$ denotes the characteristic function of the set $\mathcal{D}_{y,z}$ in \eqref{Dsetdef}. 

\begin{lemma}\label{lem1}
Let $f:[1,\infty)\to \mathbb{C}$ be integrable. For $x\ge 1$, $y\ge 1$, $z\ge 1$, we have
\begin{equation}\label{lem1eq1}
 \int_1^{x} \sum_{n\le x/t} \chi_{yt,z}(n) f(nyt) \frac{dt}{t} = \int_1^{x} A(t,y,z) f(yt) \frac{dt}{t}.
\end{equation}
If $f(t)\ll 1/(t\log 2t)$, then 
\begin{equation}\label{lem1eq2}
\int_1^\infty \sum_{n\ge 1} \chi_{yt,z}(n) f(nyt) \frac{dt}{t} = \int_1^\infty A(t,y,z) f(yt) \frac{dt}{t}.
\end{equation}
\end{lemma}

\begin{proof}
We have
\begin{equation*}
\begin{split}
 \int_1^{x} \sum_{n\le x/t} \chi_{yt,z}(n) f(nyt) \frac{dt}{t} 
& = \sum_{n \le x \atop {F(n)\le xy \atop P^-(n)>z}} \int_{\max(1,F(n)/ny)}^{x/n}  f(nyt) \frac{dt}{t}\\
& = \sum_{n \le x \atop { F(n)\le xy  \atop P^-(n)>z}}\int_{\max(n,F(n)/y)}^{x}  f(yt) \frac{dt}{t}\\
& =  \int_{1}^{x} \sum_{n \le t \atop {F(n)\le yt  \atop P^-(n)>z}} f(yt) \frac{dt}{t}
=  \int_{1}^{x} A(t,y,z) f(yt) \frac{dt}{t},
\end{split}
\end{equation*}
which establishes \eqref{lem1eq1}. 
If $f(t)\ll 1/(t\log 2t)$, then  $\lim_{x\to \infty} \int_1^{x} A(t,y,z) f(yt) \frac{dt}{t}$ exists, since $A(t,y,z)\le A(yt) \ll yt /\log 2 yt $.
Lemma \ref{UB} and partial summation yield 
$$  \sum_{n> x/t} \chi_{yt,z}(n) f(nyt) \ll \frac{\log yt}{yt \log xy},$$ 
so that 
$$  \int_1^{x} \sum_{n\le x/t} \chi_{yt,z}(n) f(nyt) \frac{dt}{t} =  \int_1^{x} \sum_{n\ge 1} \chi_{yt,z}(n) f(nyt) \frac{dt}{t} + O\left(\frac{1}{\log xy}\right).$$
Equation \eqref{lem1eq2} now follows from \eqref{lem1eq1} by taking the limit as $x\to \infty$. 
\end{proof}

\begin{corollary}\label{Acor1}
For $y\ge z$, we have
$$
\int_1^\infty A(t,y,z)\Pi(yt) \frac{dt}{t^2}
=y\int_y^\infty A(t/y,y,z)\Pi(t) \frac{dt}{t^2} = \Pi(z).
$$
\end{corollary}
\begin{proof}
The first and second expression are equal by a change of variables. That the first expression equals $\Pi(z)$
follows from  \eqref{lem1eq2} with $f(t)= \frac{1}{t}\Pi(t)$ and Lemma \ref{A2} with $y$ replaced by $yt$, i.e.
$
\sum_{n\ge 1} \frac{\chi_{yt,z}(n)}{n}\Pi(nyt)=\Pi(z).
$
\end{proof}

\begin{proof}[Proof of Theorems \ref{thmAD} and \ref{thmA}]
Assume $x\ge y \ge z \ge 3/2$. 
We write $n=mpr$, where $F(n)=p^2 r$, $F(m)\le p$, $F(r)\le p^2 r$ and $P^-(r)\ge p$. 
Sorting the integers $n$ in $\mathcal{A}(x,y,z) \setminus \mathcal{D}(x,y,z)$ according to $p$, we have
$$
(A-D)(x,y,z) = \sum_{z<p\le \sqrt{xy}} \sum_{{m<p/y \atop F(m)\le p} \atop P^-(m)>z} \sum_{{r\le xy/p^2 \atop F(r)\le p^2 r} \atop P^-(r)\ge p} 1
= \sum_{y<p\le \sqrt{xy}} A'(p/y,y,z) D(xy/p^2,p^2,p-1),
$$
where, for $x> 1$, $y\ge z \ge 3/2$,
\begin{equation}\label{firstUB}
A'(x,y,z):= |\{n<x: F(n)\le xy, \ P^-(n)>z\}|= 1+O\left(\frac{x\log y}{ \log xy \log z}\right),
\end{equation}
by Lemma 5 of \cite{AEDD}, if $x\ge z$. When $z>x$, then $A'(x,y,z)=1$.   
Let $M:=\max(y,(xy)^{1/5})$. With $E(x,y,z)$ as in \eqref{Edef}, the contribution from primes $p > M$ to $A(x,y,z)-D(x,y,z)$ is,
by Lemma \ref{UB},
$$
 \sum_{M<p\le \sqrt{xy}} \left(1 + O  \left(\frac{p \log y}{y \log p \log z}\right) \right) \left(  1 + O\left(\frac{x y}{p^2 \log xy}\right)\right) 
= \pi(\sqrt{xy},M)+O(E(x,y,z)).
$$
The contribution from the term $1$ in the estimate \eqref{Dasymp} and from primes $p\le M$ is 
$$
 \sum_{y<p\le M} A'(p/y,y,z) =\sum_{y<p\le M}\left(1+ O\left(\frac{p \log y}{y \log p \log z}\right)\right)  
= \pi(M,y)+ O(E(x,y,z)).
$$
With \eqref{Dasymp} we obtain
$$
A(x,y,z)-D(x,y,z) =\pi(\sqrt{xy},y)+ \sum_{y<p\le M} A'(p/y,y,z) \frac{c_{p^2,p-1} xy}{p^2 \log xy} + O(E(x,y,z)).
$$
Since $c_{p^2,p-1} \asymp 1$ by \eqref{cyz3}, extending the last sum to all $p>y$ adds $\ll E(x,y,z)$. 
Thus
$$
A(x,y,z)-D(x,y,z) =\pi(\sqrt{xy},y)+  \frac{\alpha_{y,z} x}{\log xy}+O(E(x,y,z)),
$$ 
where
$$
\alpha_{y,z} := \sum_{p>y} A'(p/y,y,z) \frac{c_{p^2,p-1} y}{p^2 } \ll \frac{1}{\log z} .
$$
The estimate \eqref{thmADeq} now follows from Theorem \ref{AF} with $\tau_{y,z}=\beta_{y,z}+ \alpha_{y,z}$. 
We will establish \eqref{taueq2} below. 

Define
$a_{y,z}:= \alpha_{y,z} + c_{y,z}$. 
Together with \eqref{Dasymp2} we conclude that, for $x\ge y \ge z \ge 3/2$, 
\begin{equation}\label{Aabserror}
A(x,y,z) =  1+ \pi(\sqrt{xy},z)+\frac{a_{y,z} x}{\log xy}+ O\left(E(x,y,z)+\frac{x\log (1+y/z) \log^2 y}{\log^3 xy \log z}\right).
\end{equation}
We have $a_{y,z}\asymp \frac{\log(1+y/z)}{\log z}$, which is \eqref{ateq4}, to be established below as a consequence of \eqref{Aabserror}.
This implies that, for $x\ge 1$, $y\ge z\ge 3/2$, 
\begin{equation}\label{secondUB}
A(x,y,z) =1+ \min(\pi(x,z), \pi(\sqrt{xy},z))+ O\left(\frac{x\log (1+y/z)}{\log xy \log z}\right).
\end{equation}
Note that \eqref{Aasymp1} is implied by \eqref{secondUB} if $y>x^{1/10}$, say. Thus we may assume $y\le x^{1/10}$.
Repeating the above argument with \eqref{secondUB} in place of \eqref{firstUB}, and assuming $y\le x^{1/10}$,  we find that
$$
A(x,y,z) =  1+ \pi(\sqrt{xy},z)+ \frac{a_{y,z} x}{\log xy}+ O\left(\frac{x\log(1+ y/z)}{\log^2 xy \log z}+\frac{x\log (1+y/z) \log^2 y}{\log^3 xy \log z}\right),
$$
which yields \eqref{Aasymp1}, assuming \eqref{ateq4}, which will be established below from \eqref{Aabserror}.  

When $1\le y \le z\le x$, the result follows since in that case $A(x,y,z)= A(xy/z,z,z)$ and $a_{y,z}=a_{z,z} y/z$. 

The estimate \eqref{Aabserror} clearly implies that for $y\ge z$ fixed and $x\to \infty$,
$$ \int_1^x A(t,y,z)\frac{dt}{t} \sim \frac{a_{y,z} x}{\log x} \quad (x \to \infty).$$
On the other hand, Lemma \ref{lem1} with $f(t)=1$ yields
\begin{equation*}
 \int_1^x A(t,y,z)\frac{dt}{t} 
=\int_{1}^{x}D(x/t,yt,z) \frac{dt}{t}.
\end{equation*}
With the estimate \eqref{Dasymp}, we find that the last integral is
$$
\frac{x}{\log x} \int_1^\infty c_{yt,z} \frac{dt}{t^2} + O_{y,z}\left(\frac{x}{\log^2 x}\right).
$$
Thus, for $y\ge z$,
\begin{equation}\label{adef1}
 a_{y,z} =  \int_1^\infty c_{yt,z} \, \frac{dt}{t^2}.
\end{equation}
To show \eqref{adef2}, we substitute \eqref{tdensecon} into \eqref{adef1} to get
$$ a_{y,z}= C \int_1^\infty \sum_{n\ge 1} \frac{\chi_{yt,z}(n)}{n}  \bigl( \Sigma(nyt)-\Sigma(z) - \log n\bigr) \Pi(nyt) 
\frac{dt}{t^2}.
$$
Now replace $\Sigma(nyt)-\Sigma(z)-\log n$ by $(\Sigma(nyt)+\gamma-\log(nyt)) +(\log(yt)-\gamma - \Sigma(z))$, and observe that
\begin{multline*}
 \int_1^\infty \sum_{n\ge 1} \frac{\chi_{yt,z}(n)}{n} (\log yt-\gamma-\Sigma(z)) \Pi(nyt)
\frac{dt}{t^2} 
\\ =\Pi(z)\int_1^\infty (\log yt-\gamma-\Sigma(z))  \, \frac{dt}{t^2}
=\Pi(z)(1+\log y -\gamma - \Sigma(z)),
\end{multline*}
since
$
 \sum_{n\ge 1} \frac{\chi_{yt,z}(n)}{n}  \Pi(nyt)=\Pi(z),
$
by Lemma \ref{A2}.
Thus,
\begin{equation*}
a_{y,z}  = C\Pi(z)(1+ \log y-\gamma-\Sigma(z))+ C \int_1^\infty \sum_{n\ge 1} \frac{\chi_{yt,z}(n)}{n}g(nyt) \frac{dt}{t^2}.
\end{equation*}
By Lemma \ref{lem1}, with $f(u) = \frac{g(u)}{u}$, the last integral equals
$$
  \int_1^\infty A(t,y,z) g(yt) \frac{dt}{t^2}.
$$ 
This concludes the proof of \eqref{adef2}. 

The estimates \eqref{ateq1} and \eqref{ateq2} follow from \eqref{adef2}, Corollary \ref{Acor1} and \eqref{etadef}.

The estimate \eqref{ateq4} follows from  \eqref{adef2} and the fact that, for $t\ge 1$, 
$$
-0.12 <\gamma-\log 2 <\Sigma(t)+\gamma-\log t \le \Sigma(3)+\gamma-\log 3 < 0.73.
$$
Indeed, from \eqref{adef2} we have
$$
\frac{  a_{y,z}}{C \Pi(z)}   =1+ \log(y/z) - (\Sigma(z)+\gamma-\log z)+  \frac{1}{\Pi(z)}\int_1^\infty A(t,y,z) g(yt) \frac{dt}{t^2}.
$$
The bounds on $\Sigma(t)+\gamma-\log t$ imply that
$$
1+\log(y/z) - 0.73 -0.12 < \frac{  a_{y,z}}{C \Pi(z)}  < 1 + \log(y/z) +0.12 + 0.73 .
$$
Thus, for $y\ge z \ge 1$, 
$$
 C \Pi(z)(0.15+\log(y/z)) < a_{y,z} < C \Pi(z) (1.85+\log(y/z)).
$$
When $y<z$, \eqref{ateq4} follows from $a_{y,z}=a_{z,z}y/z$. 

To deduce a relationship between $\tau_{y,z}$ and $a_{y,z}$, we equate the estimates for $A(x,y,z)$ of Theorems \ref{thmAD} and \ref{thmA}
and use Theorem \ref{duv} to estimate $d(u,v)$. This shows that
$C\Pi(z) \log(y/z) +\tau_{y,z} = a_{y,z}$.
The estimate \eqref{taueq2} now follows from \eqref{ateq2}.
\end{proof}

\section{Algorithms for computing $c_{y,z}$ and $a_{y,z}$}\label{secacomp}

\subsection{The numerical computation of $c_{y,z}$}
From \eqref{tdensecon} we have
$$
\frac{c_{y,z}}{C} = \sum_{n\ge 1}\frac{\chi_{y,z}(n)}{n} \bigl(\Sigma(yn)-\Sigma(z) -\log n\bigr) \Pi(yn) = \sum_{n\le N} + \sum_{n>N} =: S_1+S_2,
$$
say.
We calculate $S_1$ on a computer. Let
$$
\varepsilon(N) : =  \sum_{n>N}\frac{\chi_{y,z}(n)}{n} \Pi(yn) = \Pi(z) - \sum_{n\le N}\frac{\chi_{y,z}(n)}{n} \Pi(yn),
$$
by Lemma \ref{A2}. The last expression allows us to calculate $\varepsilon(N)$ on a computer. 
It follows that the contribution from $n>N$ satisfies
$$
 (\log y - \gamma - \Sigma(z)  - \mathcal{E}(yN)) \varepsilon(N) \le S_2 \le  (\log y - \gamma - \Sigma(z)+ \mathcal{E}(yN)) \varepsilon(N)  ,
$$
where $\mathcal{E}(x)$ is defined in \eqref{etadef}.
The values in Table \ref{tab1} now follow from a table of values of $\mathcal{E}(2^k)$ for $24\le k \le 38$ in \cite[Lemma 4]{CFP}.

\subsection{The numerical computation of $a_{y,z}$}
Since $a_{y,z}= a_{z,z}y/z$ if $y<z$, we may assume $y\ge z \ge 1$. To estimate the tail of the integral in \eqref{adef2}, we write
$$ \Biggl| \int_N^\infty A(t,y,z) g(yt)   \frac{dt}{t^2}  \Biggr| \le  \mathcal{E}(yN) \, \varepsilon(N),
$$
where $\mathcal{E}(x)$ is as in \eqref{etadef} and 
$$
 \varepsilon(N) :=  \int_N^\infty A(t,y,z) \Pi(yt)    \frac{dt}{t^2}
=\Pi(z)- \int_1^N A(t,y,z)  \Pi(yt)  \frac{dt}{t^2},
$$
by Corollary \ref{Acor1}. The last expression allows us to find $\varepsilon(N)$ on a computer.  Let
$$ J(N)= \int_1^N A(t,y,z)   g(yt)    \frac{dt}{t^2},
$$
which we can also find on a computer.
Thus
$$        a_{y,z} =C\Pi(z)(1-\gamma+ \log y -\Sigma(z))+C(J(N)+ Q(N)) ,$$
where
$|Q(N)| \le   \mathcal{E}(yN) \, \varepsilon(N).$
The values in Table \ref{tab2} now follow from a table of values of $\mathcal{E}(2^k)$ for $24\le k \le 38$ in \cite[Lemma 4]{CFP}.
With $N=2^{33}$ and $y=z=1$, we find that $ a=a_{1,1}= 1.53796...$, as claimed in Theorem \ref{thm1}.

\section{Proof of Theorem \ref{corBrec}}\label{secproofcorBrec}

De la Vall\'ee Poussin \cite{Poussin} showed that $\sum_{p \le x}\{\frac{x}{p }\} =(1-\gamma +o(1)) \frac{x}{\log x}$.
His proof can easily be adapted to obtain an explicit error term.
\begin{lemma}\label{pflem}
For $x\ge 2$, we have
$$
\sum_{p \le x}\left\{\frac{x}{p }\right\} =  (1-\gamma) \frac{x}{\log x} +  O\left(\frac{x}{(\log x)^2}\right).
$$
\end{lemma}

\begin{lemma}\label{BpAp}
For every prime $p$,
$$
|\{n\in \mathcal{B}(x): P^{-}(n)=p \}| = A(x/p,p,p) \qquad (x\ge p).
$$
\end{lemma}
\begin{proof}
For $x\ge p$, 
$$
p \mathcal{A}(x/p,p,p-1) = \{n\in \mathcal{A}(x): P^-(n)=p\} \cup \{n \in \mathcal{B}(x): P^-(n)=p\}.
$$
Thus,
\begin{multline*}
|\{n \in \mathcal{B}(x): P^-(n)=p\}| =A(x/p,p,p-1)-|\{n\in \mathcal{A}(x): P^-(n)=p\} | \\= A(x/p,p,p).
\end{multline*}
\end{proof}

Since 
$ \lfloor x \rfloor = A(x) + \sum_{n\in \mathcal{B}(x)} \left\lfloor \frac{x}{n} \right\rfloor,$
Theorem \ref{corBrec} follows from Theorems \ref{thm1} and \ref{thmBfrac}.

\begin{theorem}\label{thmBfrac}
For $x\ge 2$, we have
$$
\sum_{n\in \mathcal{B}(x)} \left\{\frac{x}{n}\right\} =\frac{(1-\gamma + \beta ) x}{\log x} + O\left(\frac{x}{(\log x)^{3/2}}\right),
$$
where 
\begin{equation}\label{betadef}
\beta = \sum_{p\ge 2} \beta_p := \sum_{p\ge 2} \left( \int_{p}^{p^2} \frac{a_{y,p-1}}{yp} dy -\sum_{2\le k \le p}  \frac{a_{kp,p-1}}{kp} \right) 
= 0.554... 
\end{equation}
and $\beta_p \ll \frac{1}{p \log p}$. More precisely, $0.554604<\beta<0.554806$. 
\end{theorem}

\begin{proof}
 We write
\begin{equation*}
\sum_{n\in \mathcal{B}(x)} \left\{\frac{x}{n}\right\} 
= \sum_{p\ge 2} \sum_{n\in \mathcal{B}(x) \atop P^-(n)=p} \left\{\frac{x}{n}\right\} =S_1+S_2+S_3,
\end{equation*}
where 
$$
S_3 :=  \sum_{\sqrt{x}<p \le x}\left\{\frac{x}{p }\right\} =  (1-\gamma) \frac{x}{\log x} +  O\left(\frac{x}{(\log x)^2}\right),
$$
by Lemma \ref{pflem}. With the trivial estimate $\{x/n\}<1$ and Lemma \ref{BpAp}, we have
\begin{equation}\label{S2UB}
S_2 :=  \sum_{U< p\le \sqrt{x}} \sum_{n\in \mathcal{B}(x) \atop P^-(n)=p} \left\{\frac{x}{n}\right\}
\le  \sum_{U< p\le \sqrt{x}} A(x/p,p,p) \ll \frac{x}{(\log x) (\log U)},
\end{equation}
since \eqref{secondUB} implies
\begin{equation*}
A(x/p,p,p)  \ll \pi(\sqrt{x}) + \frac{x}{p (\log p) (\log x)}.
\end{equation*}
It remains to estimate 
\begin{equation*}
S_1:=  \sum_{p \le U} \sum_{n\in \mathcal{B}(x) \atop P^-(n)=p} \left\{\frac{x}{n}\right\}
= \sum_{p \le U} \sum_{1\le k < p} 
\sum_{\frac{x}{(k+1)p} < m \le \frac{x}{kp} \atop F(m)\le x, \ P^-(m)\ge p} \left( \frac{x}{mp} - k \right).
\end{equation*}
The innermost sum equals
\begin{multline*}
-\int_{kp}^{(k+1)p} \left(\frac{y}{p}-k\right) dA(x/y,y,p-1)  \\
= -A(x/y,y,p-1) \left(\frac{y}{p}-k\right) \Bigl|_{kp}^{(k+1)p}+\frac{1}{p}\int_{kp}^{(k+1)p} A(x/y,y,p-1) dy,
\end{multline*}
as a result of integration by parts. Theorem \ref{thmA} yields
\begin{equation}\label{Axyyp}
A(x/y,y,p-1) = \frac{a_{y,p-1} x}{y \log x} + O\left(\frac{x \log (1+\frac{y}{p}) }{y(\log x)^2 \log p}+\frac{x \log y \log (1+\frac{y}{p})}{y (\log x)^3}\right),
\end{equation}
for $ p\le y \le p^2 \le x^{1/3}$.
We find that the last sum over $m$ equals
$$
\frac{\beta_{p,k} x}{\log x} +  O\left(\frac{x \log(k+1)}{p(\log p) k (\log x)^2}+\frac{x \log p \log (k+1)}{pk (\log x)^3}\right),
$$
where
$$
\beta_{p,k} := \int_{kp}^{(k+1)p} \frac{a_{y,p-1}}{yp} dy -\frac{a_{(k+1)p,p-1}}{(k+1)p},
$$
provided $U\le x^{1/6}$.
Note that $\beta_p:=\sum_{1\le k <p} \beta_{p,k} \ll \frac{1}{p\log p}$ for each prime $p$, 
as the two innermost sums in the definition of $S_1$ amount to $<A(x/p,p,p)$, by the trivial estimate $\{x/n\}<1$, as in the case of $S_2$.
Summing over $1\le k <p$ and then over $p\le U$, the contribution from the error term is 
$$\ll \frac{x \log U}{(\log x)^2} + \frac{x (\log U)^3}{(\log x)^3},$$ 
while the main term contributes
$$
\frac{\beta x}{\log x} + O\left( \sum_{p>U} \frac{\beta_p x}{\log x}\right) = \frac{\beta x}{\log x} + O\left(  \frac{x}{(\log x)(\log U)}\right).
$$
The result now follows with $\log U = \sqrt{\log x}$. 

The numerical calculation of $\beta$ is based on formulas  \eqref{betadef} and \eqref{adef2}. The details are given in Section \ref{secnumbeta}.
\end{proof}

\section{Proof of Theorem \ref{Rcor}}\label{secRcor}

To derive Theorem \ref{Rcor} from Theorem \ref{corBrec}, we will need the following estimate. 
\begin{proposition}\label{Rprop}
Let $q\ge 1$ be a fixed integer. For $x\ge 2$, we have
$$
\sum_{n\in \mathcal{B}(x) \atop \frac{x}{q+1} < n \le \frac{x}{q}} \frac{1}{n} = \frac{\mu_q}{\log x} + O_q\left(\frac{1}{\log^2 x}\right),
$$
where 
\begin{equation}\label{muqdef}
\mu_q = \log(1+1/q) + \sum_{p > q} \frac{1}{p} \left( a_{qp,p-1}-a_{(q+1)p,p-1} + \int_{qp}^{(q+1)p} a_{y,p-1} \frac{dy}{y} \right).
\end{equation}
We have $0.401720 <\mu_5 < 0.401815$ and $0.197932 < \mu_{11} < 0.197989$.
\end{proposition}

We first show how Theorem \ref{Rcor} follows from Proposition \ref{Rprop} and Theorem \ref{corBrec}.
\begin{proof}[Proof of Theorem \ref{Rcor}]
Let $\mathcal{B}'(x)$ be the set obtained from $\mathcal{B}(x)$, by replacing every $n\in (\frac{x}{6},\frac{x}{5}]\cap \mathcal{B}(x)$
by $\{2n,3n,5n\}$, and every $n\in (\frac{x}{12},\frac{x}{11}]\cap \mathcal{B}(x)$ by $\{3n,4n,5n,7n,11n\}$.
Let $n_1\in (x/6,x/5] \cap \mathcal{B}(x)$, $m_1\in \{2,3,5\}$, $n_2 \in (x/12,x/11]\cap \mathcal{B}(x)$ and $m_2 \in \{3,4,5,7,11\}$. 
Then, $m_i n_i \notin \mathcal{B}(x)$ for $i=1,2$. 
Also $n_1 m_1 \not= n_2 m_2$. Indeed, if $n_1 m_1 = n_2 m_2$ then $m_2>m_1$ and $m_1|m_2 n_2$. 
If $(m_1,m_2)=(2,4)$ then $2|n_1$, otherwise $m_1|n_2$.
Either of these is impossible since $n_i P^-(n_i)>x$ implies  $P^-(n_1) >5$ and $P^-(n_2) >11$.

Since $\mathcal{B}(x)$ has the $\lcm$ property \eqref{lcmprop}, so does $\mathcal{B}'(x)$. 
Theorem \ref{corBrec} and Proposition \ref{Rprop} yield
$$
\sum_{n\in \mathcal{B}'(x)} \frac{1}{n} = \sum_{n\in \mathcal{B}(x)} \frac{1}{n} 
+ \frac{1}{30}\sum_{n\in \mathcal{B}(x) \atop \frac{x}{6} < n \le \frac{x}{5}} \frac{1}{n} 
+ \frac{79}{4620}\sum_{n\in \mathcal{B}(x) \atop \frac{x}{12} < n \le \frac{x}{11}} \frac{1}{n} =1 -\frac{\kappa}{\log x} +O\left(\frac{1}{(\log x)^{3/2}}\right),
$$
where 
\begin{equation}\label{kappadef}
\kappa = \delta - \mu_5 \frac{1}{30}-\mu_{11}\frac{79}{4620} =0.543...
\end{equation}
More precisely, $0.543595<\kappa<0.543804$.
\end{proof}

\begin{proof}[Proof of Proposition \ref{Rprop}]
This proof is similar to that of Theorem \ref{thmBfrac}. 
Let $q\ge 1$ be a fixed integer. Since $nP^-(n)>x$ for all $n\in \mathcal{B}(x)$, $n\in \mathcal{B}(x) \cap (x/(q+1),x/q]$ implies $P^-(n)>q$. 
We write
\begin{equation*}
\sum_{n\in \mathcal{B}(x) \atop \frac{x}{q+1} < n \le \frac{x}{q}} \frac{1}{n}
= \sum_{p>q} \sum_{{n\in \mathcal{B}(x) \atop   P^-(n)=p} \atop \frac{x}{q+1} < n \le \frac{x}{q}} \frac{1}{n} =S_1+S_2+S_3,
\end{equation*}
where 
$$
S_3 :=  \sum_{ \frac{x}{q+1} < p \le \frac{x}{q}}\frac{1}{p} =  \frac{\log(1+1/q)}{\log x} +  O\left(\frac{1}{(\log x)^2}\right),
$$
by the prime number theorem. With Lemma \ref{BpAp}, we have
$$
S_2 :=  \sum_{U< p\le \sqrt{x}} \sum_{{n\in \mathcal{B}(x) \atop   P^-(n)=p} \atop \frac{x}{q+1} < n \le \frac{x}{q}} \frac{1}{n}
\le  \sum_{U< p\le \sqrt{x}}\frac{q+1}{x} A(x/p,p,p) \ll \frac{q+1}{(\log x) (\log U)},
$$
as in \eqref{S2UB}.
It remains to estimate 
\begin{equation*}
S_1:=  \sum_{q<p \le U} \sum_{{n\in \mathcal{B}(x) \atop   P^-(n)=p} \atop \frac{x}{q+1} < n \le \frac{x}{q}} \frac{1}{n}
= \sum_{q<p \le U} 
\sum_{\frac{x}{(q+1)p} < m \le \frac{x}{qp} \atop F(m)\le x, \ P^-(m)\ge p}\frac{1}{pm}.
\end{equation*}
The innermost sum equals
\begin{multline*}
-\int_{qp}^{(q+1)p} \frac{y}{px} dA(x/y,y,p-1)  \\
= -  \frac{y}{px} A(x/y,y,p-1)  \Bigl|_{qp}^{(q+1)p}+\frac{1}{px}\int_{qp}^{(q+1)p} A(x/y,y,p-1) dy,
\end{multline*}
as a result of integration by parts. By \eqref{Axyyp}, this equals
$$
\frac{\mu_{q,p} }{\log x} +  O_q\left(\frac{1}{p(\log p) (\log x)^2}+\frac{ \log p }{p (\log x)^3}\right),
$$
where
$$
\mu_{q,p} := \frac{1}{p}\left( a_{qp,p-1} - a_{(q+1)p,p-1} + \int_{qp}^{(q+1)p} \frac{a_{y,p-1}}{y} dy \right),
$$
provided $U\le x^{1/6}$.
Note that $\mu_{q,p}\ll_q \frac{1}{p\log p}$ for each prime $p$, 
as the innermost sum in the definition of $S_1$ amounts to $<\frac{(q+1)}{x}A(x/p,p,p)$, as in the case of $S_2$.
Summing over primes $p$ with $q<p\le U$, the contribution from the error term is 
$$\ll_q \frac{1 }{(\log x)^2} + \frac{ \log U}{(\log x)^3},$$ 
while the main term contributes
$$
\frac{\mu_q - \log(1+1/q)}{\log x} + O_q\left( \sum_{p>U} \frac{\mu_{q,p} }{\log x}\right) =\frac{\mu_q - \log(1+1/q)}{\log x} + O_q\left(  \frac{1}{(\log x)(\log U)}\right).
$$
The result now follows with $U=x^{1/6}$. 

The numerical calculation of $\mu_q$ is based on formulas  \eqref{muqdef} and \eqref{adef2}. The details are given in Section \ref{secnummuq}.
\end{proof}

\section{Proof of Theorems \ref{thmHasy} and \ref{thmHasy2} }\label{secproofHasy}

\begin{proof}[Proof of Theorem \ref{thmHasy}]
Let $ \tau <1$. It is easy to see that each $n\le x$ is either a multiple of a member of $\mathcal{B}_\tau(x)$, or a member of $\mathcal{A}(\tau x)$, but not both.
We write
 $$
\sum_{n\in \mathcal{B}_\tau(x)} \frac{1}{n} = \sum_{n\in \mathcal{B}(\tau x)} \frac{1}{n} + \sum_{\tau x< p \le x} \frac{1}{p}=S_1+S_2,
$$
say.
Note that $\tau \in \mathcal{T}(x)$ implies $\tau x\to \infty$, or else $S_2>1$. 
Thus $S_1 \sim 1$, by Theorem \ref{corBrec}. 
It follows that $S_2 = o(1)$ and $\log \tau = o(\log x)$. 
By Theorem \ref{corBrec} and the prime number theorem,
$$
\sum_{n\in \mathcal{B}_\tau(x)} \frac{1}{n}
= 1-\frac{\delta}{\log \tau x} + \frac{\log 1/\tau}{\log x} + O\left(\frac{1}{(\log x)^{3/2}}+ \frac{\log^2 1/\tau}{\log^2 x} \right).
$$
This implies that
$$
\tau_0(x) := \min \mathcal{T}(x) = e^{-\delta} +O(1/\sqrt{\log x})
$$
and, by Theorem \ref{thm1},
$$H^*(x) = A(x\tau_0(x))= \frac{a e^{-\delta} x}{\log x} + O\left(\frac{x}{(\log x)^{3/2}}\right). $$
\end{proof}

\begin{proof}[Proof of Theorem \ref{thmHasy2}]
As in the proof of Theorem \ref{thmHasy}, we find that, for fixed $\mu$, we have
$\tau_0(x,\mu):=\min \mathcal{T}(x,1+\mu/\log x) = e^{-\delta - \mu} + O(1/\sqrt{\log x})$. 

For the second claim we follow the proof of Ruzsa \cite[Thm. I]{Ru}.
For the upper bound we use again $H(x,z)\le H^*(x,z)$.
Since $S_2 \le z\le Z$,  any $\tau \in \mathcal{T}(x,z)$ must satisfy $x\tau \ge x^\eta$ for a suitable $\eta>0$ depending on $Z$ only. 
Thus $S_1=1+O(1/\log x)$ and we want to choose $\tau$ as small as possible such that 
$S_2 \le z - S_1$, that is 
$$
S_2 = \log \frac{\log x}{\log x\tau} +O(1/\log x) \le z-1 +O(1/\log x).
$$
Denoting this value of $\tau$ by $\tau_0$, it follows that $x\tau_0 \asymp  x^{e^{1-z}}$ and 
$$
H(x,z) \le A(x\tau_0) \ll \frac{x \tau_0}{\log x \tau_0} \asymp \frac{x^{e^{1-z}}}{\log x}.
$$ 
For the lower bound, we let $y:= H(x,z) \log x$ and we may assume that $y<x$, or else there is nothing to prove. 
Let $\mathcal{S}$ be the optimal set in the definition of $H(x,z)$. 
The primes in $(y,x]$ must belong to $\mathcal{S}$, with at most $y/\log x$ exceptions. Thus
$$
\sum_{n\in \mathcal S \atop n>y} \frac{1}{n} \ge \sum_{y<p\le x} \frac{1}{p} - \frac{y}{\log x} \cdot \frac{1}{y} 
= \log \frac{\log x}{\log y} + O\left(\frac{1}{\log x}\right). 
$$
As before, we must have $y\ge x^\eta$ for a suitable $\eta>0$, or else the last sum is too large. 
Let $U(y,\mathcal{S}):=|\{n\le y: s\in \mathcal{S} \Rightarrow s \nmid n\}|.$ Then 
$U(y,\mathcal{S})+ \sum_{n\in \mathcal{S}} \lfloor y/n \rfloor \ge \lfloor y \rfloor$. Thus
$$
\sum_{n \in \mathcal{S} \atop n\le y} \frac{1}{n} \ge 1 - \frac{1+U(y,\mathcal{S})}{y} \ge 1-\frac{1+H(x,z)}{y} \ge 1-O\left(\frac{1}{\log x}\right).
$$
Thus,
$$
z \ge \sum_{n\in \mathcal{S}} \frac{1}{n} \ge 1 +  \log \frac{\log x}{\log y} + O\left(\frac{1}{\log x}\right),
$$
that is
$
y \gg x^{e^{1-z}},
$
which is the desired lower bound.  
\end{proof}

\section{The numerical calculation of $\beta$}\label{secnumbeta}

Let 
$$
\eta(t) =  \Sigma(t) + \gamma - \log t,
$$
\begin{equation}\label{epmxdef}
\mathcal{E}^+(x)=\sup_{t\ge x} \eta(t), \quad 
\mathcal{E}^-(x)= \inf_{t\ge x}\eta(t), \quad
\mathcal{E}(x)=\sup_{t\ge x} |\eta(t)|.
\end{equation}

From \eqref{betadef} and \eqref{adef2} we have
$$
\beta_p    =\frac{C\Pi(p-1)}{p} Q_p + \frac{C}{p} R_p,
$$
where
$$
Q_p = \int_{p}^{p^2} \frac{ 1 -\gamma +\log y -\Sigma(p-1)  }{y}dy  - \sum_{k=2}^p   \frac{ 1 -\gamma +\log kp -\Sigma(p-1)}{k}
$$
and $R_p$ is given by \eqref{Rpdef}. We evaluate the integral and write the result as
$$
Q_p =  (1 - \gamma +\log p -\Sigma(p-1))  \left(\log p - \sum_{k=2}^p \frac{1}{k}\right)
+\frac{\log^2 p}{2} - \sum_{k=2}^p \frac{\log k}{k}.
$$
When $p$ is large, we estimate $Q_p$ as 
$$
Q_p =(1-\eta(p-0))(1-\gamma  -\varepsilon_p) -\gamma_1  -\xi_p, 
$$
where $0\le \varepsilon_p \le \frac{1}{2p}$, $\gamma_1$ is the Stieltjes constant
$$
\gamma_1 := \lim_{x\to \infty} \left(\sum_{k=2}^x \frac{\log k}{k} - \frac{\log^2 x}{2} \right) =-0.072815845483676724860...
$$
and, by Euler summation,
$$
0\le \xi_p := \int_p^\infty \{x\} \frac{\log x -1}{x^2} dx \le \frac{\log p}{p}.
$$
We want to estimate
$$
Q:=\sum_{p\ge 2} \frac{C Q_p \Pi(p-1)}{ p } = \sum_{2 \le p\le N}+ \sum_{p>N}.
$$
We calculate the first sum directly. For the second sum, 
since
$
\sum_{p>N} \frac{\Pi(p-1)}{p} = \Pi(N),
$
we have 
$$
\sum_{p>N} \frac{C Q_p \Pi(p-1)}{ p } = C \Pi(N)\Bigl((1-\delta^*_N) (1-\gamma  -\varepsilon^*_N) -\gamma_1 -\xi^*_N\Bigr),
$$
where $ \mathcal{E}^-(N)\le \delta^*_N \le \mathcal{E}^+(N)$, $0 \le \varepsilon^*_N \le \frac{1}{2N}$, $0\le \xi^*_N \le \frac{\log N}{N}$. 
With $N=2^{36}$ and \cite[Table 1]{CFP} we obtain
 $$0.4232907784<Q<0.4232910253.$$

It remains to estimate 
$$
R:=\sum_{p\ge 2} \frac{C R_p}{ p } = \sum_{2 \le p\le N}+ \sum_{p>N},
$$
where $R_p$ equals
\begin{equation}\label{Rpdef}
R_p=\int_p^\infty \sum_{k=2}^p  \left(\int_{(k-1)p}^{kp}\Bigl( A(t/y,y,p-1)  - A(t/kp,kp,p-1)\Bigr)dy \right) g(t) \frac{dt}{t^2} .
\end{equation}
Thus,
$$
R_p=\int_p^\infty \sum_{k=2}^p  \left(\int_{(k-1)p}^{kp} 
\sum_{\frac{t}{kp} < n \le \frac{t}{y} \atop F(n)\le t, \ P^-(n) \ge p}  1 \ dy \right) g(t) \frac{dt}{t^2} .
$$
Switching the order of the inner sum and integral, we get
$$
R_p=\int_p^\infty \sum_{k=2}^p  \left(
\sum_{\frac{t}{kp} < n \le \frac{t}{(k-1)p} \atop F(n)\le t, \ P^-(n) \ge p}\left( \frac{t}{n}-(k-1)p\right) \right) g(t) \frac{dt}{t^2} .
$$
Note that for $n$ in the given range, $k-1 = \lfloor \frac{t}{np} \rfloor$, and hence $\frac{t}{n} - (k-1)p = p\{ \frac{t}{np} \}$.
Thus,
$$
R_p= p \int_p^\infty 
\sum_{\frac{t}{p^2} < n \le \frac{t}{p} \atop F(n)\le t, \ P^-(n) \ge p}\left\{ \frac{t}{np}\right\} g(t) \frac{dt}{t^2} .
$$
Since $\{u\}<1$, Corollary \ref{Acor1} yields
$$
R_p = \eta^*(p) \, p \int_p^\infty (A(t/p,p,p-1)-A(t/p^2,p^2,p-1)) \Pi(t) \frac{dt}{t^2} = \eta^*(p)(1-1/p)\Pi(p-1),
$$
where $\mathcal{E}^-(p)\le  \eta^*(p) \le \mathcal{E}^+(p)$. 
This shows that the contribution from $p>N$ satisfies
$$
C  \mathcal{E}^-(N) \Pi(N)  \le \sum_{p>N} \frac{C R_p }{ p } \le C  \mathcal{E}^+(N) \Pi(N) .
$$
When $p\le N$, the contribution to $R_p$ from $t>N$ is
$$
=\delta^*(N) \,  p \int_N^\infty (A(t/p,p,p-1)-A(t/p^2,p^2,p-1))  \Pi(t) \frac{dt}{t^2} = \delta^*(N) \varepsilon(N,p),
$$
where $  \mathcal{E}^-(N)  \le  \delta^*(N)\le   \mathcal{E}^+(N)$ and, by Corollary \ref{Acor1},
$$
\varepsilon(N,p) =(1-1/p) \Pi(p-1) - p \int_p^N ( A(t/p,p,p-1)-A(t/p^2,p^2,p-1))  \Pi(t) \frac{dt}{t^2} ,
$$
the exact value of which can be found on a computer. When $p\le N$, the contribution to $R_p$ from $t\le N$ is 
$$
\tilde{R}_p:= p \int_p^N 
\sum_{\frac{t}{p^2} < n \le \frac{t}{p} \atop F(n)\le t, \ P^-(n) \ge p}\left\{ \frac{t}{np}\right\} g(t) \frac{dt}{t^2} .
$$
We switch the order of the sum and integral and break up the integral into unit intervals to get
$$
\tilde{R}_p=p\sum_{n\le N/p \atop P^-(n)\ge p} \sum_{j=\max(F(n),np)}^{\min(N-1,np^2-1)} \int_j^{j+1}\left(\frac{t}{np} - \left\lfloor \frac{j}{np}\right\rfloor \right)
\left(\Sigma(j)+\gamma-\log t\right) \Pi(j) \frac{dt}{t^2},
$$
which can be evaluated with antiderivatives to obtain exact values.
Thus, for $p\le N$, 
$$
\tilde{R}_p +  \mathcal{E}^-(N) \varepsilon(N,p) < R_p < \tilde{R}_p +   \mathcal{E}^+(N) \varepsilon(N,p).
$$
Combining everything we get  
$$
 R< C \sum_{p\le N} \frac{\tilde{R}_p}{p}+  C \mathcal{E}^+(N) \left(  \sum_{p\le N} \frac{\varepsilon(N,p)}{p}  + \Pi(N)\right)
$$
and
$$
 R> C \sum_{p\le N} \frac{\tilde{R}_p}{p}+  C \mathcal{E}^-(N) \left(  \sum_{p\le N} \frac{\varepsilon(N,p)}{p}  + \Pi(N)\right)
$$

We use the fact that for $2\le t \le 2^{38}$, $\eta(t)>0$. As a result, we have 
\begin{equation}\label{Eminus}
\mathcal{E}^-(x)\ge -\mathcal{E}(2^{38}) \ge -0.00000305 \qquad (x \ge 2),
\end{equation}
by \cite[Table 1]{CFP}.
With $N=2^{21}$, $\mathcal{E}^+(N)\le 0.00105$ and \eqref{Eminus} we obtain 
$$
0.131313700<R<0.131514383.
$$
With the estimate for $Q$ we obtain
$$
0.554604 <\beta = Q+R < 0.554806.
$$

\section{The numerical calculation of $\mu_q$}\label{secnummuq}

Recall the notation \eqref{epmxdef}.
From \eqref{muqdef} we have 
\begin{equation*}
\mu_q = \log(1+1/q) + \sum_{p > q} \mu_{q,p},
\end{equation*}
where
\begin{equation}\label{muqp}
\mu_{q,p} = \frac{1}{p} \left( a_{qp,p-1}-a_{(q+1)p,p-1} + \int_{qp}^{(q+1)p} a_{y,p-1} \frac{dy}{y} \right).
\end{equation}
We insert the formula \eqref{adef2}, that is, 
\begin{equation}\label{ayzrem}
 a_{y,z}=C\Pi(z) \left( 1 -\gamma +\log y -\Sigma(z)\right)+  C\int_1^\infty A(t,y,z) g(yt) \frac{dt}{t^2}  .
\end{equation}
The contribution from the term $C\Pi(z) \left( 1 -\gamma +\log y -\Sigma(z)\right)$ in \eqref{ayzrem} to $\mu_{q,p}$ is $\frac{C \Pi(p-1)}{ p} Q_{q,p}$, where
\begin{multline*}
Q_{q,p}= \log \frac{q}{q+1} + \int_{qp}^{(q+1)p} (1-\gamma+\log y -\Sigma(p-1))\frac{dy}{y}\\
= \frac{1}{2}\left( \log^2(q+1)-\log^2 q\right)+\log(1+1/q) (\log p - \gamma -\Sigma(p-1)).
\end{multline*}
Thus, the contribution from the first term in \eqref{ayzrem} to $\mu_q$ is
$$
\sum_{p>q} \frac{C \Pi(p-1)}{ p} Q_{q,p} = \frac{C}{2}\left( \log^2(q+1)-\log^2 q\right)\Pi(q)+ C\log(1+1/q)S_q
$$
where
$$
S_q= \sum_{p>q} \frac{\Pi(p-1)}{p} (\log p -\gamma -\Sigma(p-1)) = \sum_{q< p \le N} \frac{\Pi(p-1)}{p} (\log p -\gamma -\Sigma(p-0)) + E(N).
$$
The last sum can be calculated on a computer and the error term satisfies
$$
- \Pi(N) \mathcal{E}^+(N) \le E(N)\le -\Pi(N) \mathcal{E}^-(N).
$$
Multiplying by $C \log (1+1/q)$ we obtain
\begin{equation}\label{muqe1}
 -C \log (1+1/q)  \Pi(N) \mathcal{E}^+(N) \le C \log (1+1/q)E(N) \le -C \log (1+1/q)  \Pi(N) \mathcal{E}^-(N).
\end{equation}
It remains to estimate the contribution form the integral term in \eqref{ayzrem} to $\mu_q$.  We write its contribution to $\mu_{q,p}$ as $\frac{C}{p}R_{q,p}$, where
\begin{multline*}
R_{q,p} = \int_1^\infty  A(t, qp,p-1) g(qpt)\frac{dt}{t^2} -\int_1^\infty A(t,(q+1)p,p-1) g((q+1)pt)\frac{dt}{t^2}\\
+ \int_{qp}^{(q+1)p} \int_1^\infty A(t,y,p-1) g(yt) \frac{dt}{t^2} \frac{dy}{y}.
\end{multline*}
After a change of variables this is
\begin{multline*}
R_{q,p} = pq \int_1^\infty  A(t/(pq), qp,p-1) g(t)\frac{dt}{t^2} \\-p(q+1)\int_1^\infty A(t/(p(q+1)),(q+1)p,p-1) g(t)\frac{dt}{t^2}
\\+ \int_{qp}^{(q+1)p} \int_1^\infty A(t/y,y,p-1) g(t) \frac{dt}{t^2}dy.
\end{multline*}
We rewrite this as
\begin{multline*}
R_{q,p} = pq \int_1^\infty  \Bigl[A(t/(pq), qp,p-1) -A(t/(p(q+1)),(q+1)p,p-1)  \Bigr] g(t)\frac{dt}{t^2} \\
+ \int_1^\infty \int_{qp}^{(q+1)p} \Bigl[ A(t/y,y,p-1)-  A(t/(p(q+1)),(q+1)p,p-1)  \Bigr] dy \,  g(t) \frac{dt}{t^2}.
\end{multline*}
The integral over $y$ is
$$
\int_{qp}^{(q+1)p} \sum_{\frac{t}{p(q+1)}<m\le \frac{t}{y} \atop F(m)\le t, \ P^-(m)\ge p} 1 \  dy
= \sum_{\frac{t}{p(q+1)}<m\le \frac{t}{pq} \atop F(m)\le t, \ P^-(m)\ge p} (t/m - pq) .
$$
It follows that
$$
R_{q,p} = \int_{qp}^\infty \sum_{\frac{t}{(q+1)p} < m \le \frac{t}{qp} \atop F(m) \le t, \ P^-(m) \ge p} \frac{t}{m} g(t) \frac{dt}{t^2} =\int_{qp}^N + \int_N^\infty.
$$
For $p\le N/q$, we write
$$
R^*_{q,p}:=\int_{qp}^N = \sum_{P^-(m) \ge p \atop m \le \frac{N}{qp}} \frac{1}{m} \int_{\max(mqp,F(m))}^{\min(m(q+1)p,N)} g(t) \frac{dt}{t},
$$
which we calculate on a computer more efficiently by first storing values of $\int_1^n g(t) \frac{dt}{t}$ in a table. 
The contribution of this main term from $p\le N/q$ to $\mu_q$ is
$$
\sum_{p\le N/q} \frac{C R^*_{q,p}}{p}.
$$
Since $\frac{t}{m}\le (q+1)p$,  the tail of the integral satisfies
\begin{multline*}
 \int_N^\infty
= p (q+1)\eta^*_{q,p}(N) \int_N^\infty (A(t/pq,pq,p-1)-A(t/p(q+1),p(q+1),p-1)) \Pi(t) \frac{dt}{t^2}\\ 
=: p (q+1) \eta^*_{q,p}(N)\varepsilon_{q,p}(N),
\end{multline*}
where $ \mathcal{E}^-(N) \le  \eta^*_{q,p}(N) \le  \mathcal{E}^+(N)$.
By Corollary \ref{Acor1}, $\varepsilon_{q,p}(N)$ equals
$$
\frac{\Pi(p-1)}{pq}-\frac{\Pi(p-1)}{p(q+1)} -  \int_1^N (A(t/pq,pq,p-1)-A(t/p(q+1),p(q+1),p-1)) \Pi(t) \frac{dt}{t^2},
$$
which we can find on a computer, similar to the calculation of $R^*_{q,p}$ above.
 Thus the error $E_{\text{tail}}$ for $\mu_q$ from the tail of the integral and $p\le N/q$ amounts to 
\begin{equation}\label{muqe2}
 C (q+1) \mathcal{E}^-(N) \sum_{p\le N/q} \varepsilon_{q,p}(N) \le E_{\text{tail}} \le  C (q+1) \mathcal{E}^+(N) \sum_{p\le N/q} \varepsilon_{q,p}(N).
\end{equation}

For $p>N/q$ we have the estimate (again by Cor. \ref{Acor1} and since $t/m \le (q+1)p$),
$$
R_{q,p}\le \mathcal{E}^+(N) p(q+1) \left(\frac{\Pi(p-1)}{pq}-\frac{\Pi(p-1)}{p(q+1)} \right)=\frac{\mathcal{E}^+(N)\Pi(p-1)}{q},
$$
and, similarly,
$$
R_{q,p}\ge \frac{\mathcal{E}^-(N)\Pi(p-1)}{q}.
$$
Thus, the contribution from $p>N/q$ to $\mu_q$ is
\begin{equation}\label{muqe3}
 \frac{C \Pi(N/q) \mathcal{E}^-(N)}{q} \le  \sum_{p>N/q} \frac{C R_{q,p}}{p} \le \frac{C \Pi(N/q) \mathcal{E}^+(N)}{q}.
\end{equation}
The total error can be bounded by combining \eqref{muqe1}, \eqref{muqe2}, and \eqref{muqe3}.

With \cite[Table 1]{CFP}, the estimates for $\mu_5$ and $\mu_{11}$ in Proposition \ref{Rprop}  
now follow with $N=2^{36}$ in \eqref{muqe1}, while $N=2^{21}$ in  \eqref{muqe2} and \eqref{muqe3}.

\section*{Acknowledgments}
The author is grateful for many helpful suggestions provided by Eric Saias and by the anonymous referee.

\end{document}